\documentclass{amsart}
\usepackage{amsmath}
\usepackage{enumerate}
\usepackage{float}
\usepackage{comment}
\usepackage{fancyvrb}
\usepackage{caption}
\usepackage{tikz-cd}
\usepackage{subcaption}

\usepackage{mathtools}
\usepackage{scalerel}
\usepackage{amssymb}
\usepackage{hyperref}
\hypersetup{
    colorlinks=true,
    linkcolor=blue,
    filecolor=magenta,      
    urlcolor=blue,
    citecolor=blue,
}
\usepackage{color}
\usepackage{graphicx}
\usepackage{mathrsfs}

\usepackage{mathabx}
\usepackage{amsthm}
\usepackage{color}

\usepackage[font=small]{caption}
\usetikzlibrary{cd}
\usetikzlibrary{decorations.markings}

\theoremstyle{definition}
\newtheorem{prop}{Proposition}[section]
\newtheorem{thm}{Theorem}[section]

\newtheorem{lemma}{Lemma}[section]
\newtheorem{defn}{Definition}[section]
\newtheorem{cor}{Corollary}[section]
\newtheorem{rmk}{Remark}[section]
\newtheorem{ex}{Example}[section]

\def\S{\mathfrak{S}}

\def\C{\mathbb{C}}

\def\Z{\mathbb{Z}}

\def\droop{\operatorname{min-droop}}
\def\undroop{\operatorname{min-undroop}}
\def\swap{\operatorname{cross-bump-swap}}

\def\red{\mathsf{R}}
\def\yellow{\mathsf{L}}
\def\wt{\operatorname{wt}}

\def\BPD{\operatorname{BPD}}
\def\blank{\operatorname{blank}}

\def\id{\operatorname{id}}

\def\perm{\operatorname{perm}}
\def\jdt{\operatorname{jdt}}

\def\+{\includegraphics[scale=0.4]{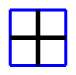}}
\def\bl{\includegraphics[scale=0.4]{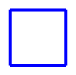}}
\def\bt{\includegraphics[scale=0.4]{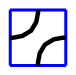}}
\def\rt{\includegraphics[scale=0.4]{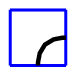}}
\def\jt{\includegraphics[scale=0.4]{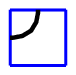}}

\def\vtile{\includegraphics[scale=0.4]{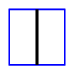}}
\def\htile{\includegraphics[scale=0.4]{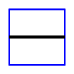}}
\def\ch{\operatorname{ch}_{\uparrow}}
\def\chd{\operatorname{ch}_{\downarrow}}

\def\mindroop{\textsf{min-droop}}
\def\maxdroop{\textsf{max-droop}}
\def\maxd{\operatorname{max-droop}}
\def\minundroop{\textsf{min-undroop}}
\def\cbswap{\textsf{cross-bump-swap}}
\def\term{\textsf{term}}
\def\nextl{\operatorname{nextl}}

\title[BPD RSK, growth diagrams, and Schubert structure constants]{Bumpless pipe dream RSK, growth diagrams, and Schubert structure constants}

\author{Daoji Huang}
\author{Pavlo Pylyavskyy}

\address{Department of Mathematics, University of Minnesota - Twin Cities,  \mbox{Minneapolis, MN, 55455}}
\email{\href{mailto:huan0664@umn.edu}{{\tt huan0664@umn.edu}}}

\address{Department of Mathematics, University of Minnesota - Twin Cities,  \mbox{Minneapolis, MN, 55455}}
\email{\href{ppylyavs@umn.edu }{{\tt ppylyavs@umn.edu}}}

\date{}

\begin{document}

\maketitle

\begin{abstract}
We introduce analogs of left and right RSK insertion for Schubert calculus of complete flag varieties. The objects being inserted are certain biwords, the insertion objects are bumpless pipe dreams, and the recording objects are decorated chains in Bruhat order. As an application, we adopt Lenart's growth diagrams of permutations to give a combinatorial rule for Schubert structure constants in the separated descent case. 
\end{abstract}

\section{Introduction}
The cohomology ring $H^*(Fl(\C^n))$ of the complete flag variety of nested vector spaces in $\C^n$ has basis given by the cohomology classes of the Schubert varieties, $[X_w]$. The \emph{Schubert structure constants} are coefficients $c_{w,v}^u$ of the Schubert classes in the expansion of the product of two Schubert classes, $[X_w]\cdot [X_v]=\sum_{u} c_{w,v}^u [X_u]$. The Schubert structure constants $c_{w,v}^u$ bear the geometric interpretation of counting the number of points in a suitable intersection of Schubert varieties determined by $u,v,w$, but its combinatorial interpretation has yet to be understood in full generality. 

A typical approach in algebraic combinatorics to study this problem is through Schubert polynomials $\S_w$, which are polynomial representatives of the Schubert classes and form a basis in the polynomial ring in variables $x_1,x_2,\cdots$. It is known that multiplication of these Schubert polynomials $\S_w\S_v=\sum_{u}c_{w,v}^u \S_u$ give rise to the same structure constants as the cohomology classes. 

An important special case is when the permutations $w$ and $v$ both have a single descent at some position $k$. In this case, the Schubert structure constants are known as the Littlewood-Richardson coefficients. There are many known combinatorial interpretations for the Littlewood-Richardson coefficients and several different directions for generalizing this case. Kogan \cite{kogan2000schubert} generalized it to the case when $v$ has a single descent at position $k$, and  $w$ either has no descents before position $k$, or no descents after position $k$. Different rules for this case has been given by Knutson-Yong \cite{knutson2004formula} (also generalized to K-theory) and Lenart \cite{lenart2010growth}. Knutson and Zinn-Justin introduced a further generalization to Kogan's case, allowing $v$ and $w$ to have ``separated descents'', and gave a puzzle-based rule \cite{allentalk}. The first author also gave a combinatorial rule for the separated descent case using tableaux and bumpless pipe dreams \cite{huang2021schubert}. The main technique there is to generalize jeu de taquin on semi-standard Young tableaux to bumpless pipe dreams.

In this paper, we employ different techniques as compared to \cite{huang2021schubert} that generalize the classical Schensted's insertion algorithm and RSK correspondence to the complete flag variety case, using bumpless pipe dreams and decorated chains in Bruhat order. Our insertion algorithms come from understanding Monk's rule for Schubert polynomials bijectively using bumpless pipe dreams. More specifically, we give two different bijective proofs of Monk's rule that generalize the row and column insertions on semi-standard Young tableaux.  The right insertion algorithm on bumpless pipe dreams is an extension of the first author's bijective proof of the one-variable version of Monk's rule \cite{huang2021bijective}, and the left insertion algorithm presented here is new. Bijective proofs of Monk's rule have been presented with pipe dreams (RC-graphs) in \cite{bergeron1993rc} and \cite{billey2019bijective}. Under the bijection given in \cite{gao2021canonical}, we conjecture that our left insertion algorithm on bumpless pipe dreams corresponds to the insertion algorithm on pipe dreams for Monk's rule given in \cite{bergeron1993rc}, and it was shown in \cite{gao2021canonical} that the algorithm in \cite{huang2021bijective} corresponds to the algorithm in \cite{billey2019bijective}. Therefore, our story about RSK could conceivably be told in terms of pipe dreams as well as bumpless pipe dreams. We work out the bumpless pipe dream version in full technical detail. 

In classical RSK correspondence, the right (row) and left (column) insertions always commute. However, this does not hold in general for our Monk's rule based left and right insertions. Theorem~\ref{thm:comm} is our major technical result regarding the commutativity of the left and right insertions.  We show that when certain conditions based on descents are satisfied, the left and right insertions commute, and this commutativity generalizes the commutativity in the usual Grassmannian case. Furthermore, we show that the commutativity in this special case can be used as a tool to obtain a combinatorial rule for the Schubert structure constants in the separated descent case. For this, we adopt Lenart's growth diagram of permutations and show that in the separated descent case, we can construct growth diagrams that propagate the local criterion for commutativity of insertions globally. Our combinatorial interpretation for Schubert structure constants $c_{w,v}^u$ in the separated descent case (Theorem~\ref{thm:rule}) is the number of certain growth diagrams determined by $u,v,w$, or equivalently, the number of certain saturated chains in (mixed) $k$-Bruhat order from $w$ to $u$ determined by $v$. We note that our technique is substantially different than the techniques used in \cite{lenart2010growth}. There, the author studied the interplay between the growth rules and a monoid structure of certan words of transpositions. In particular, this approach does not give a bijective correspondence on the level of monomials of the Schubert polynomials, but our insertion-based approach does.

From the perspective of our approach, we may understand that the difficulty of the Schubert problem manifests itself as the non-commutativity of the left and right insertion algorithms. Alternatively one can think of the difficulty as non-associativity of a potential analog of plactic monoid. Therefore, a deeper understanding of the relations of the left and right insertions could lead us to more insights to compute  Schubert products. On the other hand, the combinatorial rule for the structure constants is stated purely in terms of growth diagrams, without needing to mention bumpless pipe dreams. Perhaps its intrinsic nature deserves to be studied more extensively. 

\section{Background}
\subsection{Classical RSK correspondence}
The classical RSK correspondence is an important algorithm playing a role in Schubert calculus and  representation theory. The algorithm starts with words in alphabet $1, \ldots, n$ and produces as an output two tableaux: a semistandard insertion tableau, and a standard recording tableau. For the definitions of those terms and for the details of the algorithm we refer the reader to \cite[Chapter 7]{EC2}. What we want to emphasize is the following property of the RSK: as one considers all words with a fixed recording tableaux, the sums of monomial weights of those words give Schur functions, which represent Schubert classes in the cohomology rings of Grassmannians. In this paper we shall consider two generalizations of the classical RSK that will have a similar property. Specifically, summing weights of all biwords that have a fixed recording chain produces Schubert polynomials, see Theorems \ref{thm:leftRSK} and \ref{thm:rightRSK}

\subsection{Schubert polynomials and bumpless pipe dreams}

Bumpless pipe dreams were introduced by Lam, Lee and Shimozono in \cite{LLS}.
For a permutation $\pi\in S_{\infty}=\bigcup_{n=1}^\infty S_n$, a (reduced) \textit{bumpless pipe dream} $D$ of $\pi$ is a tiling of the square grid $\Z_{>0}\times\Z_{>0}$ using the following six tiles, $\+, \htile, \vtile, \jt, \rt, \bl$, such that the tiling forms pipes indexed by $\Z_{>0}$ where pipe $i$ travels from $(\infty,i)$ to $(\pi(i),\infty)$ towards the NE direction and that no two pipes cross twice.  

We typically draw a bumpless pipe dream in a square grid of size $n\times n$ if $\pi\in S_n$. Let $\BPD(\pi)$ denote the set of bumpless pipe dreams of $\pi$, and for $D\in\BPD(\pi)$, let  $\blank(D)$ denote the coordinates of its \bl-tiles. Then the \textit{weight} of $D$ is defined as $\wt(D):=\prod_{(i,j)\in\blank(D)}x_i.$

Bumpless pipe dreams give a combinatorial model for Schubert polynomials. It is shown in \cite{LLS} that for any permutation $\pi$, \[\S_\pi=\sum_{D\in\BPD(\pi)}\wt(D).\]
We will use this formula as our definition for Schubert polynomials in this paper.

\begin{ex}
Figure~\ref{fig:bpd31524} shows the set of bumpless pipe dreams for 31524. From these we have
\[\S_{31524}=x_1^2x_3^2+x_1^2x_2x_3+x_1^2x_2^2+x_1^3x_3+x_1^3x_2.\]

\begin{figure}[h!]
    \centering
    \includegraphics[scale=0.6]{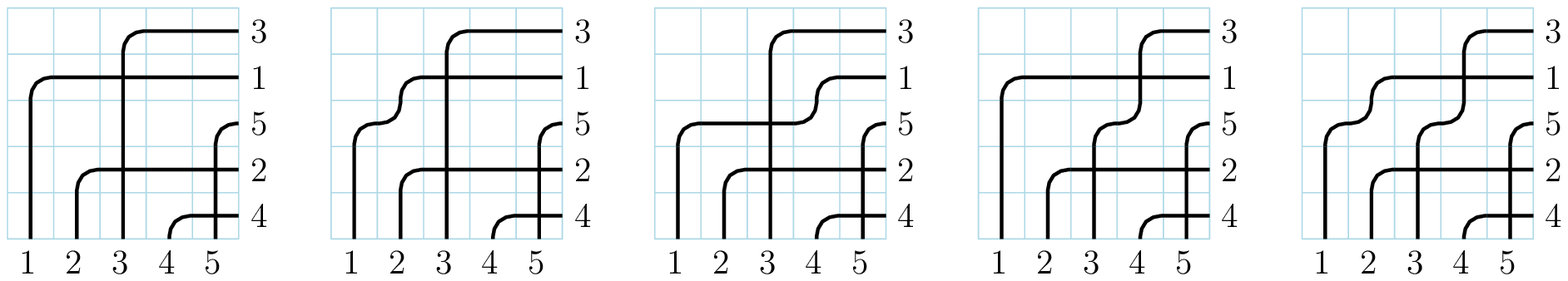}
    \caption{$\BPD(31524)$}
    \label{fig:bpd31524}
\end{figure}
\end{ex}
We state Monk's rule for Schubert polynomials \cite{monk1959geometry}. For a general permutation $\pi$ and a simple reflection $s_k$,
\[
    \S_{s_k}\S_\pi = (x_1+\cdots +x_k)\S_\pi 
    = \sum_{\substack{\pi\,t_{ab}\\a\le k<b\\ \pi\,t_{ab}\gtrdot \pi}}\S_{\pi\,t_{ab}}.
\]
The left and right insertion algorithms in Section~\ref{sec:ins} come from two different bijective interpretations of this identity.
\subsection{$k$-Bruhat order and $k$-growth diagrams}
Let $\pi$ and $\rho$ be permutations. We say $\pi$ covers $\rho$ in $k$-Bruhat order, written $\pi \gtrdot_k \rho$ if there exist $\alpha\le k<\beta$ such that $\pi = \rho\, t_{\alpha\beta}$, where $t_{\alpha\beta}$ denotes the transposition of $\alpha$ and $\beta$. A \emph{mixed $k$-chain} for a permutation $\pi$ is a chain of permutations from the identity to $\pi$, each covers the predecessor in $k$-Bruhat order for some $k$, namely,
\[(\id\lessdot_{k_1} \pi_1 \lessdot_{k_2}\cdots \lessdot_{k_\ell}\pi_\ell=\pi).\]

\begin{ex}
An example of a mixed $k$-chain for the permutation 15324 is
\[1234\lessdot_2 1324 \lessdot_3 1342 \lessdot_4 13524 \lessdot_2 15324.\]
\end{ex}

Following \cite{lenart2010growth}, we define a $k$-growth diagram to be a matrix of permutations $\pi_{i,j}$ with $0\le i \le m$, $0\le j \le n$, together with positive integers $k_i$ for each $1\le i\le m$, $l_j$ for each $1\le j\le n$,
subject to the conditions that

\begin{enumerate}[(a)]
    \item $\pi_{0,0}=\id$
    \item for each $1\le i\le m$, $\pi_{i,j}\gtrdot_{k_i} \pi_{i-1,j}$ for all $j$;
    \item for each $1\le j\le n$, $\pi_{i,j}\gtrdot_{l_j}\pi_{i,j-1}$ for all $i$;
    \item each square satisfies the property that if $x$ is the unique permutation in the open Bruhat interval $(\pi_{i-1,j-1}, \pi_{i,j})$ different from $\pi_{i,j-1}$ and $x\lessdot_{k_i} \pi_{i,j}$ and $x\gtrdot_{l_j} \pi_{i-1,j-1}$, then $\pi_{i-1,j}=x$; if $x$ does not exist, $\pi_{i-1,j}=\pi_{i,j-1}$.
    
\end{enumerate}
    The picture below shows a generic square of a growth diagram.
    \[\begin{tikzcd}
	\pi_{i-1,j} & \pi_{i,j} \\
	\pi_{i-1,j-1}& \pi_{i,j-1}
	\arrow["k_i"', no head, from=2-1, to=2-2]
	\arrow["l_j", no head, from=2-1, to=1-1]
	\arrow["k_i", no head, from=1-1, to=1-2]
	\arrow["l_j", no head, from=1-2, to=2-2]
\end{tikzcd}\]

\begin{ex}
Figure~\ref{fig:growthdiagram} shows an example of a growth diagram.

\begin{figure}[h!]
    \centering
    \[\begin{tikzcd}
	1342 & 1432 & 15324 \\
	1243 & 1342 & 13524 \\
	1234 & 1324 & 1342
	\arrow["2", no head, from=3-1, to=3-2]
	\arrow["3", no head, from=3-2, to=3-3]
	\arrow["3", no head, from=3-3, to=2-3]
	\arrow["3", no head, from=3-2, to=2-2]
	\arrow["3", no head, from=3-1, to=2-1]
	\arrow["2", no head, from=2-1, to=1-1]
	\arrow["2", no head, from=2-2, to=1-2]
	\arrow["2", no head, from=2-3, to=1-3]
	\arrow["3", no head, from=1-2, to=1-3]
	\arrow["2", no head, from=1-1, to=1-2]
	\arrow["2", no head, from=2-1, to=2-2]
	\arrow["3", no head, from=2-2, to=2-3]
\end{tikzcd}\]
    \caption{A $k$-growth diagram}
    \label{fig:growthdiagram}
\end{figure}
\end{ex}

It is shown in \cite{lenart2010growth} that growth diagrams are well-defined. By definition, given any chain
$\mathbf{c}=(\id \lessdot_{k_1}\cdots \lessdot_{k_m} w_m=w)$ from the identity to $w$, and any chain $\mathbf{d}=(w=u_0\lessdot_{l_1} u_1\lessdot_{l_2} \cdots \lessdot_{l_n} u_n=u)$ from $w$ to $u$, there is a unique growth diagram $(\pi_{i,j}, k_i, l_j)_{i\in[1,m], j\in [1,n]}$ such that $\pi_{i,0}=w_i$ and $\pi_{m,j}=u_j$ for each $i,j$. We denote the leftmost vertical chain $(\id \lessdot_{l_1} \pi_{0,1}\lessdot_{l_2}\cdots \lessdot_{l_n}\pi_{0,n})$ as $\jdt_\mathbf{c}(\mathbf{d})$. 
For example, Figure~\ref{fig:growthdiagram} shows that if $\mathbf{d}=(1342\lessdot_3 13524 \lessdot_2 15324)$ and $\mathbf{c}=(1234\lessdot_2 1324\lessdot_3 1342)$, then $\jdt_\mathbf{c}(\mathbf{d})=(1234\lessdot_3 1243\lessdot_2 1342)$.

\section{The Left and Right Insertion Algorithms}
\label{sec:ins}
 
\begin{defn}[Basic Monk moves \cite{huang2021bijective}]\label{def:basic-monk-moves}
We define   $\mindroop$, $\cbswap$, and $\term$, on bumpless pipe dreams that allow a single $\bt$. See Figure \ref{fig:basic-monk-moves}.
\vskip 0.5em

\noindent $\mindroop$: Let $(a,b)$ be the position of an \rt-turn of a pipe $p$. Note that the tile at $(a,b)$ could be a $\rt$ or $\bt$. Let $x>0$ be the smallest number where $(a+x,b)$ is not a \+, and $y>0$ be the smallest number where $(a,b+y)$ is not a \+. A $\mindroop$ at $(a,b)$ droops $p$ into $(a+x,b+y)$. We let $\droop(a,b):=(a+x,b+y)$. The map $\droop$ is only defined at tiles where the move $\mindroop$ is possible. We may also define its inverse, $\minundroop$, and the partial map on coordinates, $\undroop$, in a similar fashion.
\vskip 0.5em

Note on notation scheme: we use the notation $\mathsf{map}_D(a,b)$ to denote the resulting BPD after performing the operation defined by $\mathsf{map}$, and $\operatorname{map}_D(a,b)$ to denote a particular coordinate computed by performing $\mathsf{map}$. When the BPD $D$ is clear from context, we may drop the subscript. 
\vskip 0.5em

\noindent $\cbswap$: Suppose $(a,b)$ is a $\bt$ of pipes $p$ and $q$, and $p$ and $q$ also have a crossing at $(a',b')$. Then a $\cbswap$ move at $(a,b)$ swaps the two tiles at $(a,b)$ and $(a',b')$. We let $\swap(a,b):=(a',b')$. The map $\swap$ is only defined at tiles where the move $\cbswap$ is possible. 

\vskip 0.5em
\noindent $\term$: Suppose $(a,b)$ is a $\bt$ of non-crossing pipes $p$ and $q$. Then $\term$ at $(a,b)$ replaces the $\bt$ with a $\+$.

\begin{figure}[h!]
    \centering
    \includegraphics[scale=0.6]{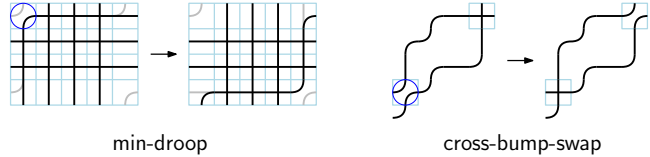}
    \caption{$\mindroop$ and $\cbswap$ at the circled coordinates}
    \label{fig:basic-monk-moves}
\end{figure}
\end{defn}

\begin{defn}
We define a \emph{biletter} as a pair of integers $(b,k)$ with $b\le k$, and for convenience write it as $b_k$.  A \emph{biword} is a word $(b_1)_{k_1}\cdots (b_\ell)_{k_\ell}$ with each entry a biletter. 
\end{defn}
\subsection{Left insertion}
Let $D\in \BPD(\pi)$ and $b_k$ be a biletter. We define the left-insertion algorithm as follows that produces $(b_k\rightarrow D)\in \BPD(\pi t_{\alpha,\beta})$ where $\alpha \le k < \beta$ as follows. Let $(i, j)$ be the position of the $\rt$ with on the row $b$ with $j$ minimal.
\begin{enumerate}[(1)]
    \item  Perform a $\mindroop$ at $(i,j)$ and let $(i_1, j_1)=\droop(i,j)$.
    \item If $(i_1,j_1)$ is a $\jt$, we check whether $i_1\le k$.
     
    \begin{enumerate}
        \item If so, let $(i_1, j_2)$ be the position of the $\rt$ with $j_2>j_1$ minimal.  Update $(i,j)$ to be $(i_1, j_2)$ and go to step (1). 
        \item Otherwise, let $(i_1, j_0)$ be the position of the $\rt$ with $j_0<j_1$ maximal. Update $(i,j)$ to be $(i_1, j_0)$ and go to step (1). 
    \end{enumerate}
    \item Otherwise $(i_1,j_1)$ is a $\bt$. 
    \begin{enumerate}
        \item If the pipe that has a $\rt$-turn at this tile exits from row $r$ with $r\le k$, we update $(i,j)$ to be $(i_1,j_1)$ and go back to step (1).
        \item Otherwise, we check whether the two pipes passing through $(i_1,j_1)$ have already crossed. If so, notice that the crossing must be below row $i_1$. In this case perform a $\cbswap$, update $(i,j)$ to be
        $\swap(i_1,j_1)$, and go back to step (1). If not, replace the tile at $(i_1,j_1)$ with a $\+$ and terminate the algorithm.
    \end{enumerate}
\end{enumerate}

The inverse of this algorithm is described as follows. Given inputs $(\pi, \rho, k, E\in \BPD(\rho))$ such that $\rho=\pi t_{\alpha \beta}$  where $\alpha\le k < \beta$ and $\ell(\pi t_{\alpha,\beta}) = \ell(\pi) + 1$, let $(i,j)$ be the position of the tile where pipes $\pi(\alpha)$ and $\pi(\beta)$ cross. We begin by replacing this tile with a $\bt$.
\begin{enumerate}[(1)]
    \item Perform a $\minundroop$ at $(i,j)$ and let $(i_1,j_1)=\undroop(i,j)$.
    \item If $(i_1,j_1)$ is a $\rt$, we check whether $i_1\le k$.

    \begin{enumerate}[(a)]
        \item If yes, we check whether there exists a $\jt$ with coordinate $(i_1,j_2)$ for some $j_2<j_1$. If it exists, take $j_2$ to be maximal, update $(i,j)$ to be $(i_1,j_2)$ and go to step (1). If it does not exist, we output the biletter $(i_1)_k$ and terminate the algorithm.  
        \item Otherwise, we let $(i_1,j_3)$ be the $\jt$ with $j_3>j_1$ minimal, update $(i,j)$ to be $(i_1,j_3)$, and go to step (1). 
    \end{enumerate}
     
    \item Otherwise $(i_1,j_1)$ is a $\bt$. 
    \begin{enumerate}
        \item If the pipe that has a $\jt$-turn at this tile exits from row $r$ with $r\le k$, update $(i, j)$ to be $(i_1,j_1)$ and go to step (1). 
        \item Otherwise, the two pipes that pass through $(i_1,j_1)$ must cross somewhere above row $i_1$. In this case we perform a $\cbswap$, update $(i,j)$ to be $\swap(i_1,j_1)$, and go to step (1).
    \end{enumerate}
\end{enumerate}

\begin{ex} \label{ex:lins}
We show two examples of left insertions.

Figure \ref{fig:lins1} shows the left insertion of the biletter $2_3$ into a BPD of 13254. 
The insertion  algorithm triggers steps (1), (3)(a), (1), (3)(a), (1), (3)(b) in order. The reverse algorithm triggers steps (1), (3)(a), (1), (3)(a), (2)(a) in order.

Figure \ref{fig:lins2} shows the left insertion of the biletter $1_2$ into a BPD of 12543. The insertion algorithm triggers steps (1), (2)(a), (1), (2)(b), (1), (3)(b) in order. The reverse algorithm triggers steps (1), (2)(b), (1), (2)(a), (1), (2)(a) in order.
\begin{figure}[h!]
    \centering
    \includegraphics[scale=0.6]{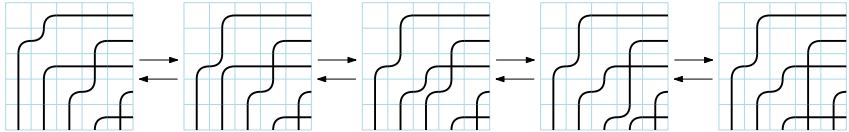}
    \caption{Left insertion of the biletter $2_3$ into a BPD of 13254. }
    \label{fig:lins1}
\end{figure}

\begin{figure}[h!]
    \centering
    \includegraphics[scale=0.6]{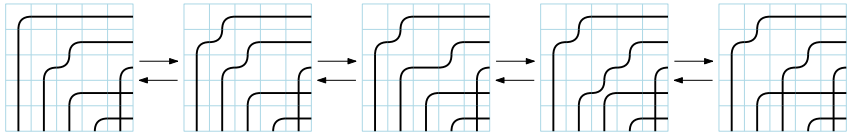}
    \caption{Left insertion of $1_2$ into a BPD of 12534.}
    \label{fig:lins2}
\end{figure}
\end{ex}

\begin{prop} \label{prop:left}
The two algorithms described above are indeed inverses of each other.
\end{prop}

\begin{proof}
For a given $k$, define {\it {activated BPD}} to be a BPD where 
\begin{itemize}
    \item one of the tiles is allowed to be a $\bt$;
    \item one of the tiles of the form $\rt$, $\jt$, or $\bt$ is considered {\it {active}};
    \item if a tile of the form $\bt$ is present, it must be the active one;
    \item the pipe in the active $\jt$ or the upper wire in the $\bt$ exit in row $\leq k$.
\end{itemize}
It is easy to check that all intermediate steps in the algorithms are activated BPDs, where the active tiles is the tile next step is being applied to. Now one can verify that the two algorithms undo each other's steps. 

First, locally  $\minundroop$ and $\mindroop$ are inverses of each other. 

Next, the first step of the inverse algorithm reverses the creation of $\+$ on the last step of the forward algorithm, and vice versa. Note that the condition on two pipes passing through active tile not having any crossings yet corresponds to the condition $\ell(\pi t_{\alpha,\beta}) = \ell(\pi) + 1$ we require at the beginning of the inverse algorithm. 

Step (2a) of the inverse algorithm, when non-terminating, is the reverse of  step (2a) of the forward algorithm, which occurs exactly when the row index of the active tile is $\leq k$. Steps (2b) of both algorithms are inverses of each other. 

Steps (3a) of both algorithms are inverses of each other. Finally, $\cbswap$ operations in steps (3b) of both algorithms are inverses of each other, since both occur when one of the pipes in the active tile exits in a row $>k$.
\end{proof}

Given a biword $Q$, we define a map $\mathcal{L}$ that maps $Q$ to a pair of bumpless pipe dream and a mixed $k$-chain in $k$-Bruhat order. We call this chain the \emph{recording chain for left insertion} of the word $Q$. \[\mathcal{L}(Q):=\begin{cases} (I,\id) & \text{ if } Q\text{ is the empty biword}\\
(b_k\rightarrow D, \mathbf{c}\lessdot_k \perm (b_k\rightarrow D))&\text{ if }Q=b_kQ' \text{ and }\mathcal{L}(Q')=(D, \mathbf{c})\end{cases}.\]
Here $I$ denotes the bumpless pipe dream for the identity permutation.

\begin{ex}
Figure \ref{fig:leftRSK} shows the insertion procedure to compute $\mathcal{L}(1_12_31_22_4)$. 
\begin{figure}[h!]
    \centering
    \includegraphics[scale=0.6]{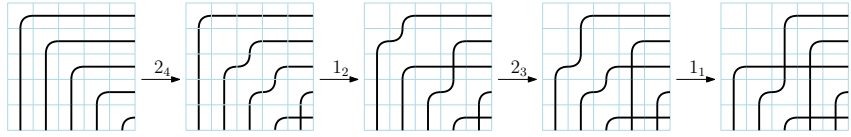}
    \caption{Left insertion of $1_12_31_22_4$}
    \label{fig:leftRSK}
\end{figure}

The recording chain is 
\[12345 \lessdot_4 12354 \lessdot_2 13254 \lessdot_3 14253 \lessdot_1 24153.
\]
\end{ex}

Conversely, given a bumpless pipe dream $D\in\BPD(\pi)$ with $\ell=\ell(\pi)$ and a mixed $k$-chain $\mathbf{c}=(\id\lessdot_{k_1} \pi_1\lessdot \cdots \lessdot_{k_{\ell}}\pi_\ell=\pi)$, we can define
$\overline{\mathcal{L}}(D,\mathbf{c}):=b_k \overline{\mathcal{L}}(D',\mathbf{c}')$ where $b_k$ and $D'$ are the results of the inverse of the left insertion  on the inputs $(\pi_{\ell-1}, \pi, k_{\ell}, D)$,  $\mathbf{c}'=(\id\lessdot_{k_1} \pi_1\lessdot \cdots \lessdot_{k_{\ell-1}}\pi_{\ell-1})$, and $\overline{\mathcal{L}}(I)=\emptyset$, the empty biword. It's clear from construction that $\mathcal{L}$ and $\overline{\mathcal{L}}$ are inverses of each other.

Define the \emph{weight} of a biword $Q=(b_1)_{k_1}\cdots (b_\ell)_{k_\ell}$ as $wt(Q):=\prod x_{b_i}^{\# b_i}$ where $\# b_i$ denotes the number of occurrences of $b_i$ in $Q$ (regardless of the subscripts $k_i$'s). The following theorem relates biwords and Schubert polynomials.
\begin{thm}
\label{thm:leftRSK}
Fix a mixed $k$-chain $\mathbf{c}$ for the permutation $\pi$. The Schubert polynomial $\S_\pi$ is the sum of the weights of all biwords $Q$ whose recording chain for left insertion is $\mathbf{c}$.
\end{thm}

\begin{proof}
For a fixed mixed $k$-chain $\mathbf{c}$ for the permutation $\pi$, $\BPD(\pi)$ and the set of biwords whose recording chains are $\mathbf{c}$ are in bijection with the via the map $\overline{\mathcal{L}}(\cdot, \mathbf{c})$. 
\end{proof}

\begin{ex}
Let $\mathbf{c}=1234\lessdot_3 1243\lessdot_2 1342\lessdot_2 1432$ be a mixed $k$-chain for 1432. Then $\{2_2 2_2 3_3, 1_2 2_2 3_3, 1_2 1_2 3_3, 2_2 2_2 1_3, 1_2 2_2 1_3\}$ is the set of biwords whose recording chain for left insertion is $\mathbf{c}$. The sum of the weights of these biwords is $x_2^2x_3+x_1x_2x_3+x_1^2x_3+x_1x_2^2+x_1^2x_2=\S_{1432}$.
\end{ex}

\subsection{Right insertion} Let $D\in \BPD(\pi)$ and $b_k$ be a biletter. We define the right-insertion algorithm that produces $(D \leftarrow b_k)\in \BPD(\pi t_{\alpha,\beta})$ where $\alpha \le k < \beta$ as follows. Let $(i, j)$ be the position of the $\rt$ with on the row $b$ with $j$ maximal.  
\begin{enumerate}[(1)]
\item  Perform a $\mindroop$ at $(i,j)$ and let $(i_1, j_1)=\droop(i,j)$. 
\item If $(i_1,j_1)$ is a $\jt$, let $(i_1,j_2)$ to be the $\rt$ with $j_2<j_1$ maximal. Update $(i,j)$ to be $(i_1,j_2)$ to be $(i,j)$ and go back to the beginning of step (1).
\item If $(i_1,j_1)$ is a $\bt$, we check whether the two pipes passing through this tile have already crossed.
\begin{enumerate}
    \item If yes, perform a $\cbswap$ and update $(i,j)$ to be $\swap(i_1,j_1)$, and go to step (1).
    \item Otherwise, we let $p$ denote the pipe of the $\rt$-turn of this tile and check if $p$ exits from row $r \le k$. If so, we update $(i,j)$ to be $(i_1,j_1)$ and go to step (1). Otherwise, we replace the tile at $(i_1,j_1)$ with a $\+$ and terminate the algorithm.
\end{enumerate}

\end{enumerate}
The inverse of this algorithm is described as follows. Given inputs $(\pi, \rho, k, E\in \BPD(\rho))$ such that $\rho=\pi t_{\alpha \beta}$  where $\alpha\le k < \beta$ and $\ell(\pi t_{\alpha,\beta}) = \ell(\pi) + 1$, let $(i,j)$ be the position of the tile where pipes $\pi(\alpha)$ and $\pi(\beta)$ cross. We begin by replacing this tile with a $\bt$.
\begin{enumerate}[(1)]
    \item Perform a $\minundroop$ at $(i,j)$ and let $(i_1,j_1)=\undroop(i,j)$.
    \item If $(i_1,j_1)$ is a $\rt$, let $(i_1,j_2)$  be the $\jt$ with $j_2>j_1$ minimal if it exists. If It does not exist, output the biletter $(i_1)_k$ and terminate the algorithm. If $j_2$ is found, update $(i,j)$ to be $(i_1,j_2)$ and go back to the beginning of step (1).
    \item If $(i_1,j_1)$ is a $\bt$, we check whether the two pipes passing through this tile have already crossed.
    \begin{enumerate}
        \item If yes, perform a $\cbswap$ and update $(i,j)$ to be $\swap(i_1,j_1)$ and go to step (1).
        \item Otherwise, we update $(i,j)$ to be $(i_1,j_1)$ and go back to step (1).
    \end{enumerate}
    
\end{enumerate}

\begin{ex} \label{ex:rins}
Figure \ref{fig:rins} shows the right insertion of the biletter $1_4$ into a BPD of 21435. The insertion algorithm triggers steps (1), (3)(a), (1), 3(b), (1), (3)(b) in this order. The reverse algorithm triggers steps (1), (3)(b), (1), (3)(a), (1), (2) in this order.
\begin{figure}[h!]
    \centering
    \includegraphics[scale=0.6]{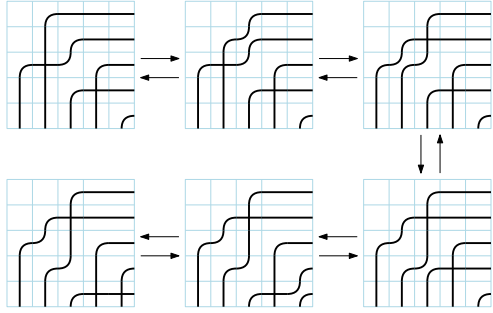}
    \caption{Right insertion of $1_4$ into a BPD of 21435.}
    \label{fig:rins}
\end{figure}
\end{ex}

\begin{prop} \label{prop:right}
The two algorithms described above are indeed inverses of each other.
\end{prop}

\begin{proof}
Define {{activated BPD}} and its active tile similarly to the proof of Proposition \ref{prop:left}. It is easy to check that all intermediate steps in the algorithms are activated BPDs, where the active tiles is the tile next step is being applied to.  

The first step of the inverse algorithm reverses the creation of $\+$ on the last step of the forward algorithm, and vice versa. Note that the condition on two pipes passing through active tile not having any crossings yet corresponds to the condition $\ell(\pi t_{\alpha,\beta}) = \ell(\pi) + 1$ we require at the beginning of the inverse algorithm. 

Step (2a) of the reverse algorithm, when non-terminating, is the reverse of  step (2a) of the forward algorithm.

Steps (3a) of both algorithms are inverses of each other. 

Finally, step (3b) of the forward algorithm is inverse of step (3b) of the inverse algorithm. Note that on every step except the very first one of the inverse algorithm the condition that the pipe of the $\rt$-turn of the active tile exits from row $\le k$ will be satisfied, since step (3a) does not change where this pipe exits, while step (3b) makes the exit index even smaller. 
\end{proof}
Parallel to the left insertion, given a biword $Q$, we define a map $\mathcal{R}$ that maps $Q$ to a pair of bumpless pipe dream and a mixed $k$-chain in $k$-Bruhat order. We call this chain the \emph{recording chain for right insertion} of the word $Q$. 

\[\mathcal{R}(Q):=\begin{cases} (I,\id) & \text{ if } Q\text{ is the empty biword}\\
( D\leftarrow b_k, \mathbf{c}\lessdot_k \perm (D\leftarrow b_k))&\text{ if }Q=Q'b_k \text{ and }\mathcal{L}(Q')=(D, \mathbf{c})\end{cases}.\]

\begin{ex}
Figure \ref{fig:leftRSK} shows the insertion procedure to compute $\mathcal{R}(1_12_31_22_4)$. 
\begin{figure}[h!]
    \centering
    \includegraphics[scale=0.6]{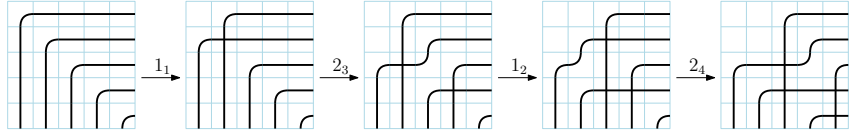}
    \caption{Right insertion of $1_12_31_22_4$}
    \label{fig:rightRSK}
\end{figure}

The recording chain is 
\[12345\lessdot_1 21345 \lessdot_3 21435 \lessdot_2 31425 \lessdot_4 31524.\]
\end{ex}

Conversely, given $D\in\BPD(\pi)$ with $\ell=\ell(\pi)$ and a mixed $k$-chain $\mathbf{c}=(\id \lessdot_{k_1}\pi_1\lessdot_{k_2}\cdots \lessdot_{k_\ell}\pi_\ell=\pi)$, we can define $\overline{\mathcal{R}}(D,\mathbf{c}):=\overline{\mathcal{R}}(D',\mathbf{c}')b_k$ where $b_k$ and $D'$ are the results of the inverse of the right insertion on the inputs $(\pi_{\ell-1},\pi,k_\ell, D)$, $\mathbf{c}'=(\id \lessdot_{k_1}\pi_1\lessdot_{k_2}\cdots \lessdot_{k_{\ell-1}}\pi_{\ell-1})$ and $\overline{R}(I)=\emptyset$, the empty biword. It's clear from the construction that $\mathcal{R}$ and $\overline{\mathcal{R}}$ are inverses of each other.  

Similar to Theorem~\ref{thm:leftRSK}, we have also the statement for  right insertion.
\begin{thm} \label{thm:rightRSK}
Fix a mixed $k$-chain $\mathbf{c}$ for the permutation $\pi$. The Schubert polynomial $\S_\pi$ is the sum of the weights of all biwords $Q$ whose recording chain for right insertion is $\mathbf{c}$.
\end{thm}

\begin{proof}
For a fixed mixed $k$-chain $\mathbf{c}$ for the permutation $\pi$, $\BPD(\pi)$ and the set of biwords whose recording chains are $\mathbf{c}$ are in bijection with the via the map $\overline{\mathcal{R}}(\cdot, \mathbf{c})$. 
\end{proof}

\begin{ex}
Let $\mathbf{c}=1234\lessdot_3 1243\lessdot_2 1342\lessdot_2 1432$ be a mixed $k$-chain for 1432. Then $\{3_3 2_2 2_2, 3_3 1_2 2_2, 3_3 1_2 1_2, 2_3 1_2 2_2, 2_3 1_2 1_2\}$ is the set of biwords whose recording chain for right insertion is $\mathbf{c}$. The sum of the weights of these biwords is $x_2^2x_3+x_1x_2x_3+x_1^2x_3+x_1x_2^2+x_1^2x_2=\S_{1432}$.
\end{ex}

\begin{rmk}
Both of our left and right RSK recover the classical RSK when we restrict to the set of biwords whose biletters all have a fixed $k$ as subscripts.
\end{rmk}

We now state our key technical result regarding commutativity of the left and right insertion algorithms and defer its technical proof until a later section. 

\begin{thm} \label{thm:comm}
Let $D\in \BPD(\pi)$. Let $x_k$ and $y_l$ be two biletters such that $l\le k$. If $\pi\neq \id$, let $d_1$ be the first descent position of $\pi$ and $d_2$ be the last descent position of $\pi$. Suppose furthermore that $k\ge d_2$ and $l\le d_1$. Then left insertion of $x_k$ commutes with right insertion of $y_l$, namely $(x_k\rightarrow D)\leftarrow y_l = x_k\rightarrow (D\leftarrow y_l)$.
\end{thm}

\begin{ex}
We show an example in Figure \ref{fig:inscommute} to illustrate the commutativity stated in Theorem \ref{thm:comm}. The right insertion of $1_2$ and the left insertion of $2_4$ into a BPD of 213645 commute.
\begin{figure}[h!]
    \centering
    \includegraphics[scale=0.6]{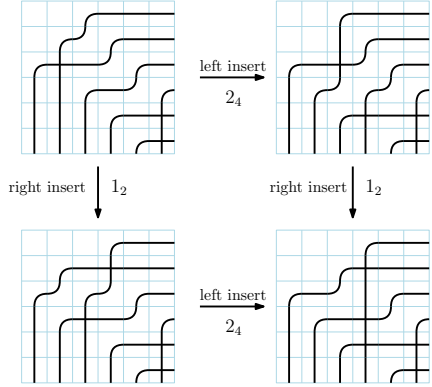}
    \caption{Commutativity of left and right insertions}
    \label{fig:inscommute}
\end{figure}
\end{ex}

\section{Growth Diagrams}
We start with a simple but useful lemma on a single square in a growth diagram. Given a permutation $\pi$, let $d_1(\pi)$ denote the first descent position of $\pi$ and $d_2(\pi)$ denote the last descent position of $\pi$. 
\begin{lemma}\label{lem:square}
Let $\pi,\sigma,\delta,\rho$ be permutations such that $\delta\lessdot_l \pi$, $\rho\lessdot_l \sigma$, $\delta\lessdot_k\rho$, and $\pi\lessdot_k\sigma$ where $l\le k$. If $\delta\neq \id$, suppose furthermore that $l\le d_1(\delta)$ and $k\ge d_2(\delta)$. Then 
\begin{enumerate}[(a)]
    \item $k\ge d_2(\pi)$  and $l\le d_1(\pi)$
    \item  $l\le d_1(\rho)$ and $k\ge d_2(\rho)$. 
\end{enumerate}
 
\end{lemma}
\begin{proof}
To show $k\ge d_2(\pi)$, we claim that $d_2(\pi)\le d_2(\delta)$. Since $\delta \lessdot_l\pi$, $\pi=\delta t_{ab} $  for some $a\le l <b$. If $b< d_2(\delta)$, then $d_2(\pi)=d_2(\delta)$. 
If $b=d_2(\delta)$, then $\pi(b)=\delta(a)<\delta(b)$ and $\pi(i)=\delta(i)$ for all $i>b$, so $d_2(\pi)\le d_2(\delta)$. 
If $b> d_2(\delta)$, suppose to the contrary that $d_2(\pi)>d_2(\delta)$. Then since $\delta$ is strictly increasing after $d_2(\delta)$, for $\pi$ to have a descent after $d_2(\delta)$ it must be the case that $a=b-1=d_2(\pi)>d_2(\delta)\ge d_1(\delta)\ge l$, which is a contradiction.

To show $l\le d_1(\pi)$, the claim is immediate if $a=l$. Suppose $a < l$. We know that $\delta(b)>\delta(a)$, and since $\delta(a)<\delta(a+1)$, we must have $\delta(a+1)>\delta(b)$. Therefore $\pi(a)=\delta(b)<\delta(a+1)=\pi(a+1)$. In this case $d_1(\pi)=d_1(\delta)$. 

The argument for $\rho$ is similar.
\end{proof}

The following statement is immediate from Lemma~\ref{lem:square}. An example of a growth diagram as described in this corollary is shown in Example \ref{ex:sepdescgrowth}.

\begin{cor}
\label{cor:allsquares}
Let $(\pi_{i,j},k_i,l_j)$ be a growth diagram with $i\le m$ and $j\le n$. Suppose the $k_i$'s are weakly increasing, the $l_j$'s are weakly decreasing, and $l_1 \le k_1$. Then for each $1\le j\le n$, $l_j$ is no greater than the first descent of $\pi_{i,j-1}$ for all $i$, and for each $1\le i \le m$, $k_i$ is no smaller than the last descent of $\pi_{i-1,j}$ for all $j$.
\end{cor}

Combining Lemma~\ref{lem:square} and Theorem~\ref{thm:comm}, we can conclude the following:

\begin{cor}
\label{cor:uniquebpd}
Let $\pi,\sigma,\rho,\delta$ be permutations and $k,l$ be integers that satisfy the conditions in Lemma~\ref{lem:square}. Then for every $D_\sigma\in\BPD(\sigma)$, there exists unique $D_\delta\in\BPD(\delta)$, $x_k$, and $y_l$ such that $(x_k\rightarrow D_\delta)\leftarrow y_l = x_k\rightarrow (D_\delta\leftarrow y_l)=D_\sigma.$
\end{cor}
\begin{proof}

Starting from $D_\sigma$, we run the inverse of the left insertion algorithm on the inputs $(\pi, \sigma,k, D_\sigma)$ and obtain a biletter $x_k$ and $D_\pi\in \BPD(\pi)$. We then run the inverse of the right insertion algorithm on the inputs $(\delta, \pi, l, D_\delta)$ and obtain a biletter $y_l$ and $D_\delta\in \BPD(\delta)$. Then by Theorem~\ref{thm:comm}, $(x_k\rightarrow D_\delta)\leftarrow y_l = x_k\rightarrow (D_\delta\leftarrow y_l)=D_\sigma.$
\end{proof}

\subsection{Two mixed chains in $k$-Bruhat order} Given a permutation $w\in S_n$, we define the up-chain and the down-chain for $w$.
\subsubsection{Up-chain}
Let $w \in S_n$ be a permutation. Let $\phi(w):=s_{w(i)}w$, where $i$ is the largest possible position for which $s_{w(i)}w\lessdot w$. Let $\kappa(w):=i-1$. Clearly, $s_{w(i)}w\lessdot_{\kappa(w)} w$. Let $\ell:=\ell(w)$ and $w_j=\phi^{\ell-j}(w)$ for  each $0\le j\le \ell$. We define a mixed chain of permutations in $k$-Bruhat order for weakly increasing values of $k$'s as follows:
\[\ch(w)=(\id=w_0\lessdot_{\kappa(w_1)} w_1\lessdot_{\kappa(w_2)}\cdots \lessdot_{\kappa(w_{\ell-1})} w_{\ell-1}\lessdot_{\kappa(w)} w=w_\ell)\]
Notice that by construction, $d_1(w)\le\kappa(w_1)\le \kappa(w_2)\le \cdots \le \kappa(w_{\ell-1})\le \kappa(w)$.
\begin{ex}
$\ch(13542) = 12345 \lessdot_3 12435\lessdot_{4}12534\lessdot_{4}12543\lessdot_4 13542.$
\end{ex}
\subsubsection{Down-chain} Let $w\in S_n$ be a permutation. Let $\psi(w)=wt_{ij}$, where $i$ is the smallest possible index with $w(i)\neq i$, and $j>i$ the smallest possible index where $wt_{ij}\lessdot w$. Let $\gamma(w)=i$. Clearly $wt_{ij}\lessdot_i w$. Let $\ell:=\ell(w)$ and $w_j=\psi^{\ell-j}(w)$ for each $0\le j\le \ell$. We define a mixed chain of permutations in $k$-Bruhat order for weakly decreasing values of $k$'s as follows:
\[\chd(w)=(\id=w_0\lessdot_{\gamma(w_1)} w_1\lessdot_{\gamma(w_2)}\cdots \lessdot_{\gamma(w_{\ell-1})} w_{\ell-1}\lessdot_{\gamma(w)} w=w_\ell)\]
Notice that by construction, $d_2(w)\ge \gamma(w_1)\ge \gamma(w_2)\ge \cdots \ge \gamma(w_{\ell-1})\ge \gamma(w)$.
\begin{ex}
$\chd(1432)=1234\lessdot_3 1243\lessdot_2 1342\lessdot_2 1432.$
\end{ex}

\subsection{Separated descent Schubert calculus} 
\begin{thm}
\label{thm:rule}
Let $w$ and $v$ be permutations such that $d_1(w)\ge d_2(v)$.
Let $\mathcal{G}_{w,v}^u$ denote the set of growth diagrams $(\pi_{i,j}, k_i, l_j)$ that satisfy the conditions that $0\le i\le \ell(w)$, $0\le j\le \ell(v)$, $\pi_{\ell(w),\ell(v)}=u$, and
\[(\pi_{0,0}\lessdot_{k_1} \pi_{1,0}\lessdot_{k_2}\cdots \lessdot_{k_{\ell(w)}}\pi_{\ell(w),0})=\ch(w),\]
\[(\pi_{0,0}\lessdot_{l_1} \pi_{0,1}\lessdot_{l_2}\cdots \lessdot_{l_{\ell(v)}}\pi_{0,\ell(v)})=\chd(v).\]
Then $c_{w,v}^u=|\mathcal{G}_{w,v}^u|.$ Equivalently, $c_{w,v}^u$ is the number of mixed $k$-chains $\mathbf{d}$ from $w$ to $u$ such that $\jdt_{\ch(w)}(\mathbf{d})=\chd(v)$.
\end{thm}

\begin{proof}
 Since $d_1(w)\ge d_2(v)$, $l_1\le k_1$ by the construction of the up-chain and the down-chain. Then by Corollary~\ref{cor:allsquares}, for each $1\le j\le \ell(v)$ we have $l_j\le d_1(\pi_{i,j-1})$ for all $i$, and for each $1\le i\le \ell(w)$ we have $k_i\ge d_2(\pi_{i-1,j})$ for all $j$. Namely, every square satisfies the conditions in Lemma~\ref{lem:square}.

We are then ready to construct a bijection between $\BPD(w)\times \BPD(v)$ with the set $\coprod_u  \BPD(u)\times \mathcal{G}_{w,v}^u$. 

Let $u$ be a permutation such that $\mathcal{G}_{w,v}^u\neq \emptyset$, and suppose $(\pi_{i,j},k_i,l_j)\in\mathcal{G}_{w,v}^u$. Let $D_u\in \BPD(u)$. By Corollary~\ref{cor:allsquares}, we may inductively apply Corollary~\ref{cor:uniquebpd} to each square in the growth diagram and find a unique $D_{i,j}\in \BPD(\pi_{i,j})$ for each $\pi_{i,j}$, a unqiue biletter $(a_{i})_{k_i}$ for each column of horizontal edges $\pi_{i-1,j}\lessdot_{k_i}\pi_{i,j}$, and a unique biletter $(b_{j})_{l_j}$ for each row of vertical edges $\pi_{i,j-1}\lessdot_{l_j}\pi_{i,j}$. Now pick any $\pi_{x,y}$ in the growth diagram. The construction ensures that starting at the identity BPD, for each lattice path consisting of north steps and east steps from $\id=\pi_{0,0}$ to  $\pi_{x,y}$ in the growth diagram, if whenever we take an east step we perform a left insertion with the biletter $a_k$ for that edge, and whenever we take a north step we perform a right insertion with the biletter $b_l$ for that edge, we will always arrive at the same $D_{x,y}\in\BPD(\pi_{x,y})$, regardless of the path taken. In particular, this uniquely determines $D_{\ell(w),0}\in \BPD(\pi_{\ell(w),0})=\BPD(w)$ and $D_{0,\ell(v)}\in\BPD(\pi_{0,\ell(v)})=\BPD(v)$.

Conversely, let $D_w=D_{\ell(w),0}\in \BPD(\pi_{\ell(w),0})=\BPD(w)$ and $D_v=D_{0,\ell(v)}\in\BPD(\pi_{0,\ell(v)})=\BPD(v)$. Compute $\overline{\mathcal{L}}(D_w,\ch(w))=(a_{\ell(w)})_{k_{\ell(w)}}\cdots(a_2)_{k_2}(a_1)_{k_1}$ and $\overline{\mathcal{R}}(D_v,\chd(v))=(b_1)_{l_1}(b_2)_{l_2}\cdots (b_{\ell(v)})_{l_{\ell(v)}}$. (Note that these computations are the contents of Theorem~\ref{thm:leftRSK} and \ref{thm:rightRSK}.) Now pick any $\pi_{x,y}$ in the growth diagram. By Theorem~\ref{thm:comm} and Corollary~\ref{cor:allsquares}, starting from the identity BPD, for each lattice path consisting of north steps and east steps from $\id=\pi_{0,0}$ to $\pi_{x,y}$, in the growth diagram, if whenever we take an east step from $\pi_{i-1,j}$ to $\pi_{i,j}$ we perform a left insertion with the biletter $(a_i)_{k_i}$, and whenever we take a north step from $\pi_{i,j-1}$ to $\pi_{i,j}$ we perform a right insertion with the biletter $(b_j)_{l_j}$, we will always arrive at the same $D_{x,y}\in \BPD(\pi_{x,y})$, regardless of the path taken. In particular, this uniquely determines $D_u=D_{\ell(w),\ell(v)}\in \BPD(u)$.

It is clear by construction that the two directions are inverses of each other. 
\end{proof}
\begin{rmk}
The choices of chains for $w$ and $v$ are not unique, and we gave one specific construction above. In general, when $w$ and $v$ have separated descents with $d_1(w)\ge d_2(v)$, as long as we choose a chain for $w$ where the $k$'s are weakly increasing, a chain for $v$ where the $l$'s are weakly decreasing, and $k_1\ge l_1$, the theorem will still hold.
\end{rmk}
\begin{ex}
\label{ex:sepdescgrowth}
In Figure~\ref{fig:sepdescgrowth} we show an example of a growth diagram for $w=13542$ and $v=1432$. 
The Schubert product expansion for $\S_w\S_v$ is
\[\S_{13542}\S_{1432}=\S_{34521}+\S_{25431}+\S_{35412}+\S_{246315}+\S_{263415}+\S_{156324}+\S_{164325}.\]
\end{ex}
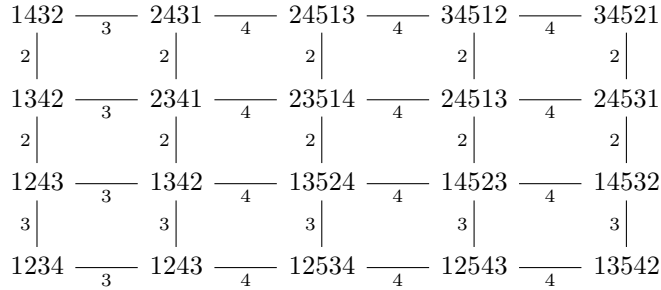
\begin{figure}[h!]
    \centering
    \[\begin{tikzcd}
	1432 & 2431 & 24513 & 34512 & 34521 \\
	1342 & 2341 & 23514 & 24513 & 24531 \\
	1243 & 1342 & 13524 & 14523 & 14532 \\
	1234 & 1243 & 12534 & 12543 & 13542
	\arrow["3", no head, from=4-1, to=3-1]
	\arrow["2", no head, from=3-1, to=2-1]
	\arrow["2", no head, from=2-1, to=1-1]
	\arrow["3", no head, from=4-2, to=3-2]
	\arrow["2", no head, from=3-2, to=2-2]
	\arrow["2", no head, from=2-2, to=1-2]
	\arrow["3", no head, from=4-3, to=3-3]
	\arrow["2", no head, from=3-3, to=2-3]
	\arrow["2", no head, from=2-3, to=1-3]
	\arrow["2", no head, from=2-4, to=1-4]
	\arrow["2", no head, from=3-4, to=2-4]
	\arrow["3", no head, from=4-4, to=3-4]
	\arrow["2", no head, from=2-5, to=1-5]
	\arrow["2", no head, from=3-5, to=2-5]
	\arrow["3", no head, from=4-5, to=3-5]
	\arrow["3"', no head, from=1-1, to=1-2]
	\arrow["3"', no head, from=2-1, to=2-2]
	\arrow["3"', no head, from=3-1, to=3-2]
	\arrow["3"', no head, from=4-1, to=4-2]
	\arrow["4"', no head, from=1-2, to=1-3]
	\arrow["4"', no head, from=2-2, to=2-3]
	\arrow["4"', no head, from=3-2, to=3-3]
	\arrow["4"', no head, from=4-2, to=4-3]
	\arrow["4"', no head, from=4-3, to=4-4]
	\arrow["4"', no head, from=3-3, to=3-4]
	\arrow["4"', no head, from=2-3, to=2-4]
	\arrow["4"', no head, from=1-3, to=1-4]
	\arrow["4"', no head, from=1-4, to=1-5]
	\arrow["4"', no head, from=2-4, to=2-5]
	\arrow["4"', no head, from=3-4, to=3-5]
	\arrow["4"', no head, from=4-4, to=4-5]
\end{tikzcd}\]
    \caption{A growth diagram for $w=13524$ and $v=1432$}
    \label{fig:sepdescgrowth}
\end{figure}

\section{Commutativity of Left and Right Insertion}
In this section, we prove Theorem~\ref{thm:comm}.

\begin{lemma}
\label{lem:leftnoswap}
Let $D\in\BPD(\pi)$. Let $x_k$ be a biletter with $k\ge d_2(\pi)$. Then left insertion of $x_k$ into $D$ never goes into step 2(b), or $\cbswap$ in step 3(b).
\end{lemma}

\begin{proof}
Since $d_2$ is the last descent of $\pi$, $D$ has no $\bl$ of $\jt$ strictly below row $d_2$. Therefore, step (2) can only be entered when the active tile is at or above $d_2$, and since $d_2\le k$, (2)(b) is never triggered. Consider the first time an active tile reaches below row $d_2$. This tile must be a $\bt$. If a $\cbswap$ is needed at this point, then there must be a $\jt$ below $d_2$, which is impossible. Therefore, once the active tile reaches below $d_2$, the procedure would enter (3)(a) for zero or more times, and terminate once it enters (3)(b).
\end{proof}

\begin{lemma}
\label{lem:droopwid}
Suppose $D\in\BPD(\pi)$, and let $y_l$ be a biletter with $l\le d_1(\pi)$. Then during the right insertion of $y_l$ into $D$, every $\mindroop$ is performed on a pipe $p$ such that $p$ does not contain any horizontal segment in a $\+$. In particular, this means that every $\mindroop$ is bounded by a width 2 rectangle. 
\end{lemma}
\begin{proof}
Suppose that during the right insertion of $y_l$ into $D$, the active tile is at $(i,j)$ and a $\mindroop$ at $(i,j)$ is the next move. Suppose the pipe that passes through the $\rt$-turn at $(i,j)$ is $p$. 
Then if $p$ passes through a horizontal segment in a $\+$, the pipe $q$ that contains the vertical segment of this tile must satisfy $p<q$ and $\pi^{-1}(q) < \pi^{-1}(p)$. Therefore, $d_1(\pi)$ must be at most $\pi^{-1}(q)$. However, during the right insertion of $y_l$, the active pipe $p$ must always exits from a row at or above $l$. Therefore, $\pi^{-1}(p)\le l\le d_1(\pi) \le \pi^{-1}(q)$, which is a contradiction.
\end{proof}

\begin{lemma}
Suppose $D\in\BPD(\pi)$, $d_1:=d_1(\pi)$, $d_2:=d_2(\pi)$, $k\ge d_2$, and $l\le d_1$. Then
the first descent of $\perm(x_k\rightarrow D)$ is no smaller than $d_1$, and the last descent of $\perm(D\leftarrow y_l)$ is not larger than $d_2$. 
\end{lemma}

\begin{proof}
Suppose $\perm(x_k\rightarrow D)=\pi\,t_{\alpha,\beta}$ where $\alpha\le k<\beta$. If the first descent of  $\pi\,t_{\alpha\beta}$ is smaller than $d_1$, then it must be the case that $\alpha<d_1$. But since $\pi\,t_{\alpha\beta}\gtrdot \pi$ and $\pi(\alpha)<\pi(d_1)$, it must be the case that $\beta \le d_1$, which is a contradiction since $\beta > k\ge d_2\ge d_1$. The second part follows by a similar argument.
\end{proof}

 By Lemma~\ref{lem:leftnoswap}, the left insertion  of $x_k$ into $D$ (or into $D\leftarrow y_l$) is a sequence of $\mindroop$s, followed by a terminal move that swaps a $\bt$ for a $\+$. 
 For right insertion, we define $\maxdroop$ as a maximal sequence of consecutive $\mindroop$s at tiles in a same column, and $\maxd$ its corresponding partial map on coordinates. See Figure~\ref{fig:minmaxdroop} for an illustration. By Lemma~\ref{lem:droopwid}, every $\maxdroop$ is bounded by a width $2$ rectangle. We consider the right insertion of $x_k$ into $D$ (or $(x_k\rightarrow D)$) as a sequence of $\maxdroop$ and $\cbswap$, followed by a terminal move that swaps a $\bt$ for a $\+$. 

\begin{figure}[h!]
    \centering
    \includegraphics[scale=0.6]{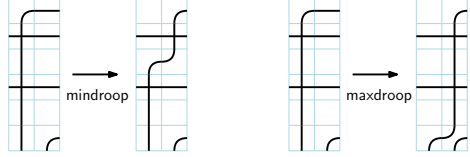}
    \caption{Comparing $\mindroop$ and $\maxdroop$}
    \label{fig:minmaxdroop}
\end{figure}
We define a \emph{doubly activated BPD} to be a bumpless pipe dream where at most two tiles are allowed to be $\bt$. For a doubly activated BPD $E$, let $E(i,j)$ denote the tile at position $(i,j)$ of $E$.  In order to reason about commutativity of the left and right insertion algorithms, we keep track of a marker for right insertion and a marker for left insertion.   We define an \emph{insertion state} to be a tuple $(E, \red, \yellow,u)$ where $E$ is a doubly activated bumpless pipe dream and $\red$ and $\yellow$ are coordinates of markers for left and right insertions such that $E(\red)=\rt$, $\jt$, or $\bt$; $E(\yellow)=\rt$ or $\bt$; $u=\rt$ or $\jt$, with constraints that when $u=\rt$, $E(\red)$ must be $\rt$ or $\bt$, and when $u=\jt$, $E(\red)$ must be $\jt$ or $\bt$. In other words, we insist that the left marker is always on a $\rt$-turn of a pipe in a $\rt$- or $\bt$-tile, whereas the right marker  is on a $\rt$-turn or $\jt$-turn of a pipe in a $\rt$, $\jt$, or $\bt$-tile. For this reason, we let $u=\rt$ or $\jt$  record whether the right marker is on the SE or NW elbow in the case when $\red$ is a $\bt$-tile. The reason we don't make this distinction for the left marker is because of Lemma~\ref{lem:leftnoswap}. We also allow $\red$ and $\yellow$ to take $\emptyset$ as a value, signaling that the corresponding right or left insertion has terminated. Finally, if $E(i,j)=\bt$, we require that $\red=(i,j)$ or $\yellow=(i,j)$. Namely, any $\bt$-tile must contain at least one marker. In all the pictures that follow, the right marker is represented with a solid dot and the left marker is represented with a circle.

\begin{ex}
We show some examples of doubly activated BPDs. For example in the second diagram, $\red=(2,4)$, $\yellow=(2,2)$, and $u=\jt$.
\ \\

\noindent%
\begin{minipage}{\linewidth}
\makebox[\linewidth]{
  \includegraphics[scale=0.6]{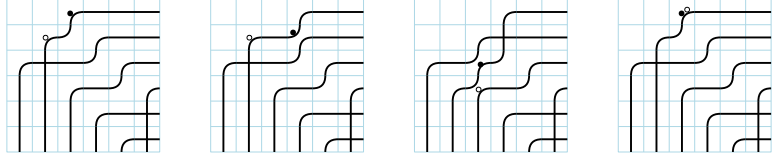}}
\captionof{figure}{Doubly activated BPDs}\label{fig:doublyactivated}
\end{minipage}
\end{ex}

We say that an insertion state $(E,\red,\yellow, u)$ is \emph{admissible} if the following forbidden configurations (Figure~\ref{fig:forbidden}) do not exist. 
\begin{enumerate}[(1)]
    \item $\red=(i,j)$, $E(i,j)=\rt$ and $E(i,j+1)=\jt$
    \item $\red=\yellow=(i,j)$, $u=\rt$, $E(i,j)=\bt$, and $E(i,j+1)=\jt$.
    \item If the right marker is on pipe $p$ and pipe $p$ occupies the $\htile$-segment in a $\+$.
\end{enumerate}
 \begin{figure}[h]
     \centering
     \begin{subfigure}{.5\textwidth}
     \centering
     \includegraphics[scale=0.9]{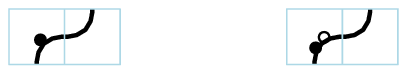}
     \caption*{ (1)}
     \label{fig:forbidden1}
      \end{subfigure}%
     \begin{subfigure}{.5\textwidth}
     \centering
     \includegraphics[scale=0.9]{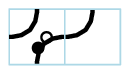}
     \caption*{ (2)}
     \label{fig:forbidden2}
     
     \end{subfigure}
    \caption{Forbidden configurations (1) and (2)}
     \label{fig:forbidden}
 \end{figure}
Note that (1) and (2) come from the definition of $\maxdroop$, and (3) come from the condition on descents for right insertion, 
Given two admissible insertion states $(E_1,\red_1,\yellow_1,u_1)$ and $(E_2,\red_2,\yellow_2,u_2)$, we identify them to be the same state if $E_1=E_2$,  $\red_1=\yellow_2$, $\yellow_1=\red_2$, $\red_1$ and $\yellow_1$ are both $\bt$, $u_1=\rt$, $u_2=\jt$ and either of the following conditions is satisfied:
\begin{enumerate}[(1)]
    \item $\red_1=(i,j)$ for some $(i,j)$, $\yellow_1=(i',j)$ with $i'>i$, and all tiles $(i'',j)$ with $i<i''<i$ are $\+$;
    \item $\red_1=(i,j)$ for some $(i,j)$ and $\yellow_1=(i,j+1)$.
\end{enumerate}
See Figure~\ref{fig:identification}.
\begin{figure}[h]
    \centering
    \includegraphics[scale=0.8]{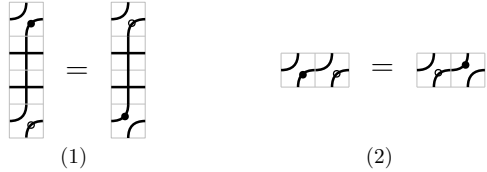}
    \caption{Identification of states}
    \label{fig:identification}
\end{figure}

We will now define a relation $\Rightarrow$ on the set of admissible insertion states that expresses transition rules. A rough overview is as follows. Given an insertion state $(E, \red, \yellow,u)$, typically we may either let the right insertion algorithm advance a step at $\red$, and transition to the state $(E',\red',\tilde{\yellow}, u')$, or let the left insertion algorithm advance a step at $\yellow$ and transition to the state $(E',\tilde{\red},\yellow', \tilde{u})$. In these cases, typically when the left (resp. right) marker advances, the right (resp. left) marker stays put, but sometimes we need to shift the other marker. Occasionally only a left move or only a right move is available from a given state. When both $\red$ and $\yellow$ are $\emptyset$, no more moves apply and the state is terminal.

If $\pi\neq\id$, let $d_1$ be the first descent of $\pi$ and $d_2$ the last descent of $\pi$. Let $l\le d_1$ and $k\ge d_2$. If $\pi=\id$, let $l\le k$. We first illustrate the cases where $\Rightarrow$ is defined by advancing the right marker. Suppose $\red\neq\emptyset$. Let $\pi$ be the permutation associated to the doubly activated BPD $E$. 
\vskip 0.5em

\textbf{Case Rr} (right marker advances from a $\rt$-turn). Suppose $\yellow=(a,b)$. If $E(\red)$ is a $\bt$ or $\rt$, $u=\rt$, let $E':=\maxdroop_E(\red)$, $R':=\maxd_E(\red)$, $u':=\jt$, and (Rr1) if $\yellow \neq \red$, let  $\tilde{\yellow}:=\yellow$, otherwise (Rr2) let $\tilde{\yellow}:=(a,b+1)$. See Figure~\ref{fig:Rr12}.

\begin{figure}[h]
    \centering
    \includegraphics[scale=0.8]{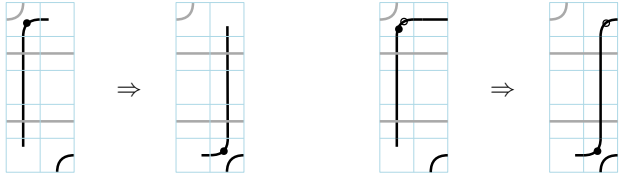}
    \caption{Case Rr1 and Rr2}
    \label{fig:Rr12}
\end{figure}
\vskip 0.5em

\textbf{Case Rj} (right marker advances from a $\jt$-turn).
\vskip 0.5em

(Rj1). Suppose $E(\red)=\jt$ where $\red=(a,b)$. Let $b'<b$ be maximal where $E(a,b')=\rt$ or $\bt$. Then $E':=E$, $\red':=(a,b')$,   $\tilde{\yellow}:=\yellow$, $u':=\rt$. See Figure~\ref{fig:Rj1}.
\begin{figure}[h]
    \centering
    \includegraphics[scale=0.8]{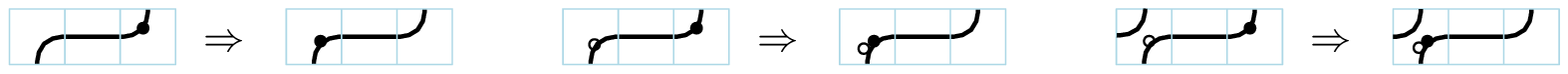}
    \caption{Case Rj1}
    \label{fig:Rj1}
\end{figure}

For (Rj2)--(Rj4) Suppose $E(\red)=\bt$, $u=\jt$, and $\yellow \neq \red$. Suppose $s$ is the pipe that contains the $\jt$-turn in $E(\red)$ and $t$ is the pipe that contains the $\rt$-turn in $E(\red)$. Notice that we must have $\pi^{-1}(s)<\pi^{-1}(t)$ because otherwise we are in the forbidden configuration (3). See Figure~\ref{fig:Rj2-4}.
\vskip 0.5em

(Rj2). Suppose $s<t$ and $\pi^{-1}(t)\le l$. Then let $E':=E$, $\red':=\red$, $\yellow':=\yellow$, and $u'=\rt$.
\vskip 0.5em

(Rj3). Suppose $s<t$ and $\pi^{-1}(t)> l$. Then let $E':=\term_E(\red)$,  $\red':=\emptyset$, $\tilde{\yellow}:=\yellow$, and $u'=\emptyset$. 
\vskip 0.5em

(Rj4). Suppose $s>t$, then the pipes $s$ and $t$ must cross at a tile SW to $\red$. Let $E':=\cbswap_E(\red)$, $\red':=\swap_E(\red)$, $\tilde{\yellow}:=\yellow$, and $u'=:\rt$. 
\begin{figure}[h]
    \centering
    \includegraphics[scale=0.8]{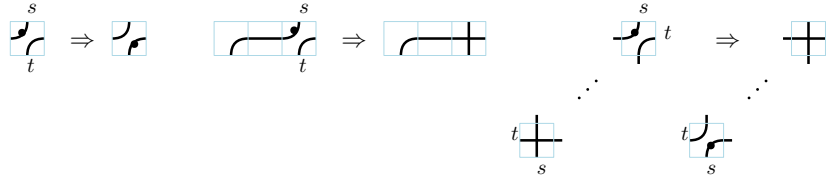}
    \caption{Case Rj2--Rj4}
    \label{fig:Rj2-4}
\end{figure}
\vskip 0.5em

(Rj5). Suppose $\red=(a,b)$, $E(\red)=\bt$, $u=\jt$, and $\yellow= \red$. Let $b'<b$ be maximal where $E(a,b')=\rt$ or $\bt$. Then $E':=E$, $\red':=(a,b')$, $\tilde{\yellow}:=\yellow$, $u':=\rt$. See Figure~\ref{fig:Rj5}.
\begin{figure}[h!]
    \centering
    \includegraphics[scale=0.8]{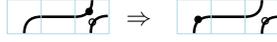}
    \caption{Case Rj5}
    \label{fig:Rj5}
\end{figure}
\vskip 0.5em

We now discuss the cases where $\Rightarrow$ is defined by advancing the left marker. 
\vskip 0.5em
\textbf{Case L} (left marker advances). Suppose $\yellow=(a,b)$.

\vskip 0.5em
(L1). Suppose $E(a,b+1)=\+$, and let $b'>b$ be smallest such that $E(a,b')\neq \+$. Suppose furthermore that $E(a,b')\neq \bt$. Note that in this case it is impossible for $\red=\yellow$ because of forbidden configuration (3). Since $E(a,b')\neq \+$, $\mindroop_E(a,b)$ is defined. Let $\tilde{E}=\mindroop_E(a,b)$ and $(c,d):=\droop_E(a,b)$. If $\tilde{E}(c,d)=\bt$, suppose $t$ is the pipe that contains the SE elbow in $\tilde{E}(c,d)$.  If
$\pi^{-1}(t)>k$ (left insertion is about to terminate), notice that in this case it is impossible to have $\red=(c,d)$, because the right marker cannot be on a pipe that exits from a row after the first descent. We then let $E':=\term_{\tilde{E}}(c,d)$, $\yellow':=\emptyset$, $\tilde{\red}:=\red$, and $\tilde{u}:=u$. In all other cases, left insertion does not terminate, and we let $E':=\tilde{E}$, $\tilde{\red}:=\red$,  $\tilde{u}:=u$, $\yellow':=\nextl(a,b)$,
where $\nextl(a,b)$ is obtained by first computing $(c,d):=\droop_E(a,b)$, then finding $d'>d$ smallest such that $E(c,d')$ is a $\rt$, and finally setting $\nextl(a,b):=(c,d')$. See Figure~\ref{fig:L1}.

\begin{figure}[h!]
    \centering
    \includegraphics[scale=0.6]{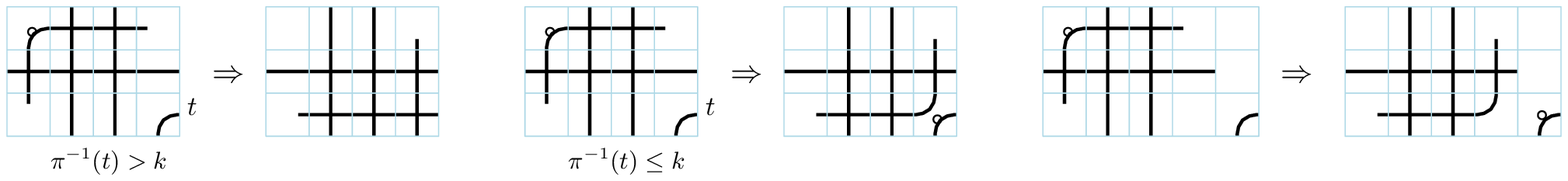}
    \caption{Case L1}
    \label{fig:L1}
\end{figure}

\vskip 0.5em
(L2). Suppose $E(a,b+1)=\htile$. Let $E':=\mindroop_E(a,b)$. If $\yellow=\red$, let $\tilde{\red}:=(a,b+1)$; otherwise $\tilde{\red}:=\red$. Also let $\tilde{u}:=u$. If left insertion does not terminate, we let $\yellow':=\nextl(a,b)$; otherwise we go through the same reasoning as in the previous case and let $E':=\term_{\tilde{E}}(c,d)$, $\yellow':=\emptyset$ where $\tilde{E}$ and $c,d$ are defined the same way as before. See Figure~\ref{fig:L2}.
\begin{figure}[h!]
    \centering
    \includegraphics[scale=0.6]{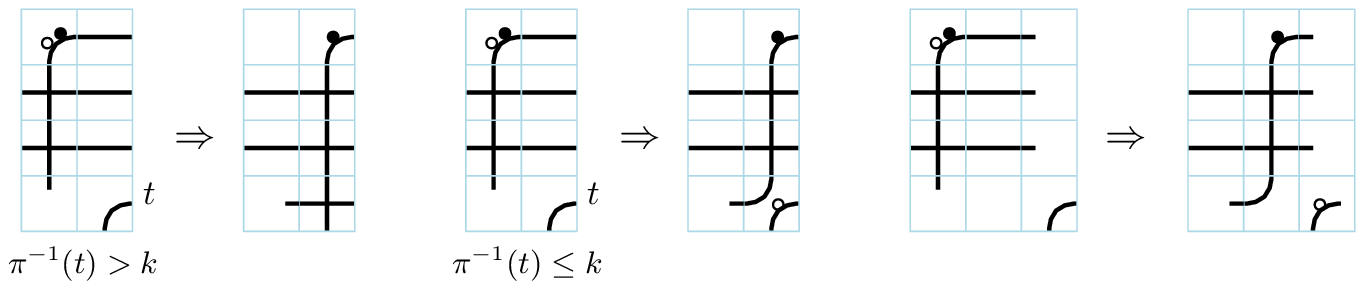}
    \caption{Case L2}
    \label{fig:L2}
\end{figure}

\vskip 0.5em
(L3). Suppose $E(a,b+1)=\jt$. By forbidden configuration (1), we know that $\red\neq \yellow$. Let $\tilde{\red}:=\red$ and $\tilde{u}:=u$. If left insertion does not terminate, we let $E':=\mindroop_E(a,b)$,  $\yellow':=\nextl(a,b)$; otherwise we let $E':=\term_{\tilde{E}}(c,d)$ where $\tilde{E}$ and $c,d$ are defined in the same way as in case (L1). See Figure~\ref{fig:L3}.
\vskip 0.5em

\begin{figure}[h!]
    \centering
    \includegraphics[scale=0.6]{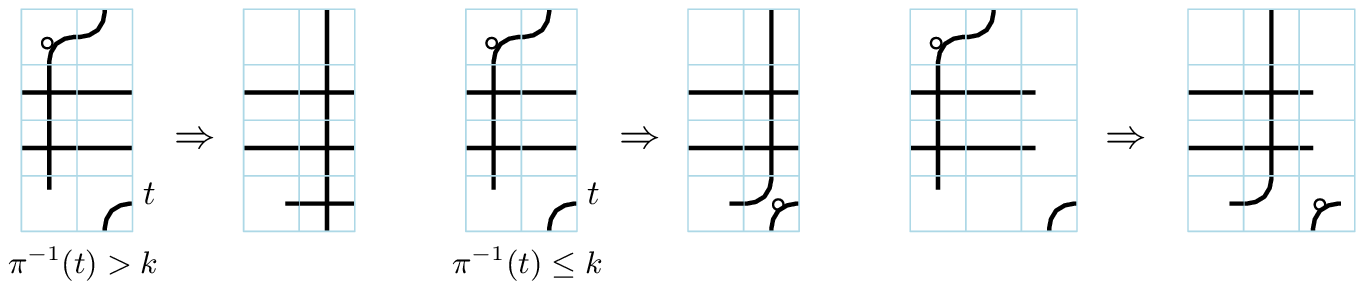}
    \caption{Case L3}
    \label{fig:L3}
\end{figure}

In the following cases, there is a unique state that $(E,\red,\yellow,u)$ can transition to.
\begin{enumerate}[(1)]
    \item Only one of $\red$ and $\yellow$ is $\neq\emptyset$.
    \item $E(\yellow)=\rt$, $E(\red)=\bt$, $\yellow$ and $\red$ are on the same row, and all tiles between $\yellow$ and $\red$ are $\+$. Only the right marker can move in this case.
    \item $E(\red)=\rt$, $E(\yellow)=\bt$, $\yellow$ is immediately to the right of $\red$. (Note that there cannot be $\+$'s between $\yellow$ and $\red$ by forbidden configuration (3).) Only the left marker can move in this case. 
    \item $E(\yellow)=\rt$, $E(\red)=\bt$, $\yellow$ and $\red$ are on the same column, and all tiles between $\yellow$ and $\red$ are $\+$. Only the right marker can move in this case.
    \item $E(\red)=\rt$, $E(\yellow)=\bt$, $\red$ and $\yellow$ are on the same column, and all tiles between $\red$ and $\yellow$ are $\+$. Only the left marker can move in this case.
\end{enumerate}
\begin{figure}[h]
    \centering
    \includegraphics[scale=0.65]{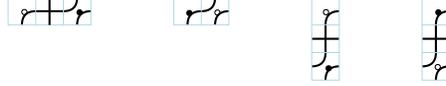}
    \caption{Only one marker is allowed to move}
    \label{fig:one-move-not-allowed}
\end{figure}
\vskip 0.5em
We are ready to prove Theorem \ref{thm:comm}.

\begin{proof}
We structure our proof using  Diamond Lemma. Let $\red:=(l,c)$ where $c$ is maximal such that $D(l,c)=\rt$. Let $\yellow=(k,b)$ where $b$ is minimal such that $D(k,b)=\rt$. Let $u=\rt$. Let $S_0$ be the initial admissible state $(D,\red,\yellow, u)$, and $\mathcal{S}:=\{S: S_0 \xRightarrow{*} S\}$, where $\xRightarrow{*}$ denotes the transitive closure of the relation $\Rightarrow$. We wish to show that there exist a unique $T\in \mathcal{S}$ such that for all $S\in\mathcal{S}$, $S\xRightarrow{*}T$. By Diamond Lemma, we need to argue that $\Rightarrow$ is well-founded (i.e., there are no infinite chains) on $\mathcal{S}$, as well as locally confluent.

To see $\Rightarrow$ is well-founded, we notice that whenever the $\yellow$ marker is moved, the area under the pipe with the $\yellow$ marker strictly decreases, and whenever the $\red$ marker is moved, either the area under the pipe with the $\red$ marker strictly decreases, or if the pipe itself doesn't move, the length of the pipe from the southern border to the $\red$ marker strictly decreases. (If two states are identified as equal, we assign the smaller statistics of the two states to this equivalence class.)

Let $(E,\red,\yellow,u)\in S$ such that neither $\red$ nor $\yellow$ is $\emptyset$.
 First we notice that if advancing the left marker at $\yellow$ and advancing the right marker at $\red$ affect disjoint sets of tiles, they obviously commute. It suffices to consider the cases when the two moves interfere with each other.

\textbf{Case ($\red \neq \yellow$ on the same row).} Suppose that $\red=(a,b)$ and $\yellow=(a,c)$. We first consider the case that $c=b+1$, $E(\red)=E(\yellow)=\bt$. This state is identified with the state $(E, (a,c), (a,b),\jt)$. Let $s$ denote the pipe that contains the $\rt$-turn in $E(a,b)$ and $t$ denote the pipe that contains the $\rt$-turn in $E(a,b+1)$. We organize our cases as follows:
\begin{itemize}
    \item $s$ and $t$ intersect (Figure~\ref{fig:Hstcross})
    \item $s$ and $t$ do not intersect
    \begin{itemize}
        \item Right move is a terminal move. Left move could be terminal or non-terminal, but the two cases are similar (Figure~\ref{fig:Hst-nocross-rterm})
        \item Right move is not a terminal move
        \begin{itemize}
        \item Left move is not a terminal move
        \begin{itemize}
            \item $E(\droop_E(a,b+1))=\bl$ (Figure~\ref{fig:Hst-nocross-noterm})
            \item $E(\droop_E(a,b+1))=\rt$ (Figure~\ref{fig:Hst-nocross-noterm2})
        \end{itemize}
        \item Left move is a terminal move. This will force the right move to be terminating in two steps. (Figure~\ref{fig:Hst-nocross-lterm})
        \end{itemize}
    \end{itemize}
\end{itemize}

\begin{minipage}{\linewidth}
    \centering
    \includegraphics[scale=0.6]{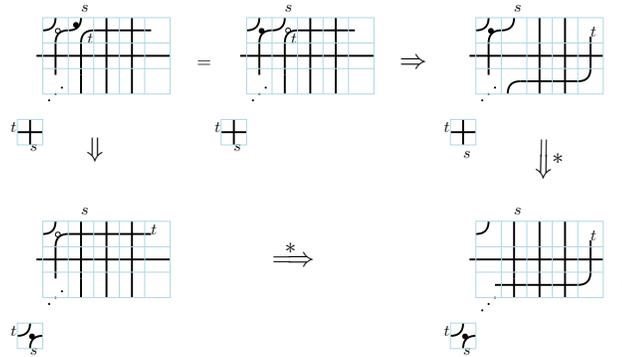}
    \captionof{figure}{$s$ and $t$ cross. Note that the two horizontal arrows could also both be terminal left moves, but this does not affect commutativity.}
    \label{fig:Hstcross}
\end{minipage}
\vskip 1em

\noindent%
\begin{minipage}{\linewidth}
    \centering
    \includegraphics[scale=0.6]{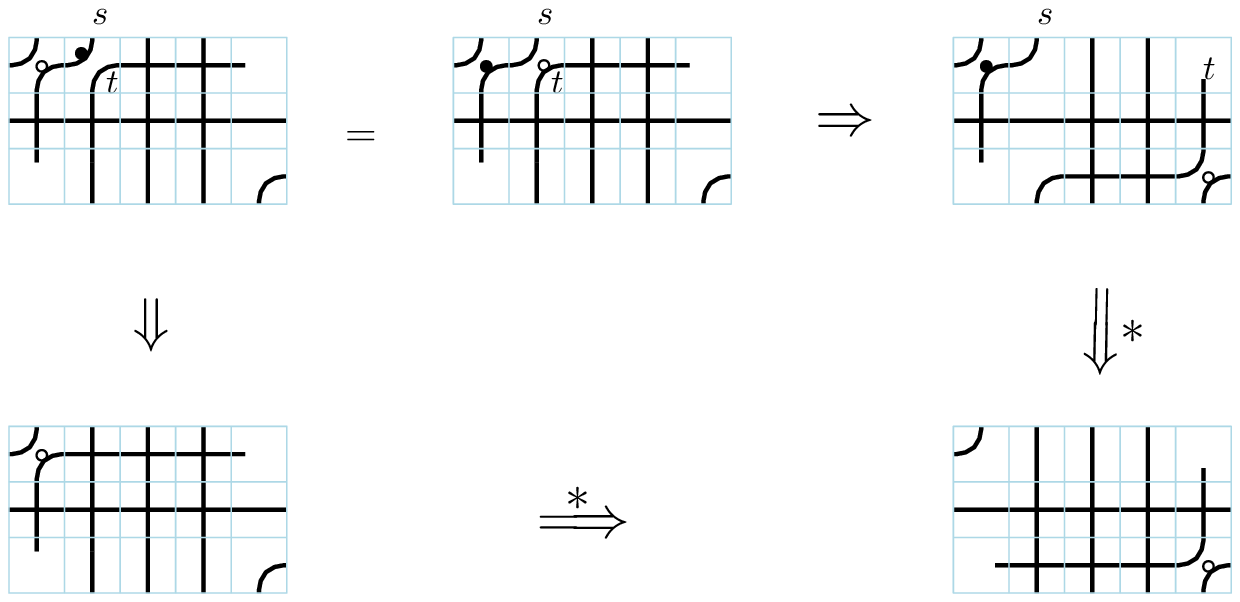}
    \captionof{figure}{ $s$ and $t$ do not cross, right move is terminal. Note that the two horizontal arrows could also both be terminal left moves, but this does not affect commutativity.}
    \label{fig:Hst-nocross-rterm}
\end{minipage}
\vskip 1em

\noindent%
\begin{minipage}{\linewidth}
    \centering
    \includegraphics[scale=0.55]{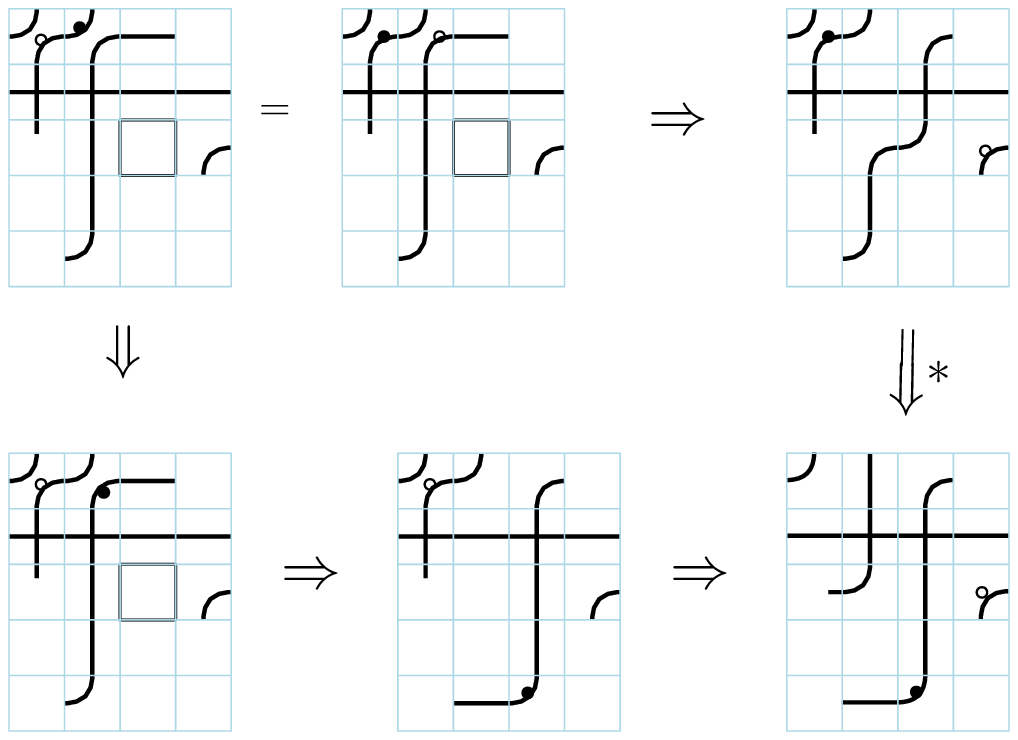}
    \captionof{figure}{$s$ and $t$ do not cross, neither left nor right move is terminal, and $E(\droop_E(a,b+1))=\bl$.}
    \label{fig:Hst-nocross-noterm}
\end{minipage}

\noindent%
\begin{minipage}{\linewidth}
    \centering
    \includegraphics[scale=0.6]{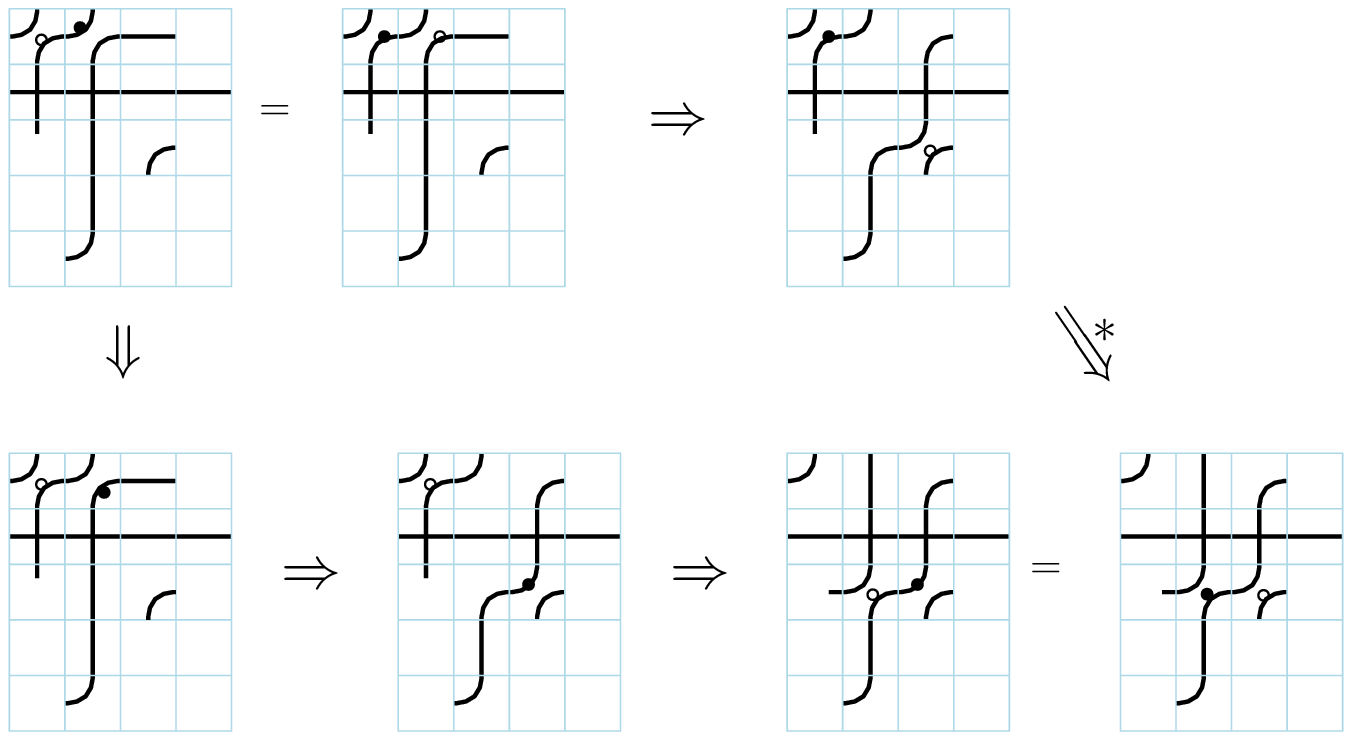}
    \captionof{figure}{$s$ and $t$ do not cross, neither left nor right move is terminal, and $E(\droop_E(a,b+1))=\rt$.}
    \label{fig:Hst-nocross-noterm2}
\end{minipage}
\vskip 1em
\noindent%
\begin{minipage}{\linewidth}
    \centering
    \includegraphics[scale=0.7]{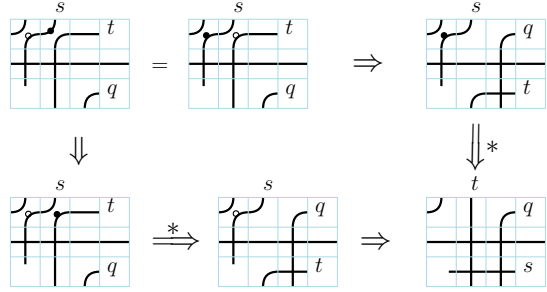}
    \captionof{figure}{$s$ and $t$ do not cross and left move is terminal. Let $q$ be the pipe that contains the $\rt$ in $E(\droop_E(a,b+1))$. It must be the case that $\pi^{-1}(q)>k\ge d_2\ge d_1\ge l$. Therefore a right move that droops $t$ into $q$ must be followed by a terminal move.}
    \label{fig:Hst-nocross-lterm}
\end{minipage}
\vskip 2em

It is easy to see that in other cases when $\red$ and $\yellow$ are on the same row, either the moves at $\red$ and $\yellow$ are disjoint, or only one of the markers is allowed to move.
\vskip 1em

\textbf{Case ($\red\neq\yellow$ on the same column).} Suppose that $\red=(a,b)$, $\yellow=(c,b)$, $c>a$. We consider the case when $E(\red)=E(\yellow)=\bt$, and the tiles strictly between $\red$ and $\yellow$ in column $c$ are $\+$. This state is identified with the state $(E,(c,b),(a,b),\jt)$. Let $s$ denote the pipe that contains the $\rt$-turn in $E(a,b)$, $t$ denote the pipe that contains the $\rt$ in $E(c,b)$. We organize the cases as follows:
\begin{itemize}
    \item $s$ and $t$ intersect (Figure~\ref{fig:Vstcross})
    \item $s$ and $t$ do not intersect
    \begin{itemize}
        \item Left move is a terminal move
        \begin{itemize}
            \item Right move is not a terminal move (Figure~\ref{fig:Vst-nocross-lterm-rnoterm})
            \item Right move is a terminal move (Figure~\ref{fig:Vst-nocross-lterm-rterm})
        \end{itemize}
        \item Left move is not a terminal move
        \begin{itemize}
            \item Right move is not a terminal move
            \begin{itemize}
                \item $E(\droop_E(c,b))=\bl$ (Figure~\ref{fig:Vst-nocross-noterm1})
                \item $E(\droop_E(c,b))=\rt$ (Figure~\ref{fig:Vst-nocross-noterm2})
            \end{itemize}
            \item Right move is a terminal move
        \end{itemize}
    \end{itemize}
\end{itemize}
\vskip 2em
\noindent%
\begin{minipage}{\linewidth}
    \centering
    \includegraphics[scale=0.6]{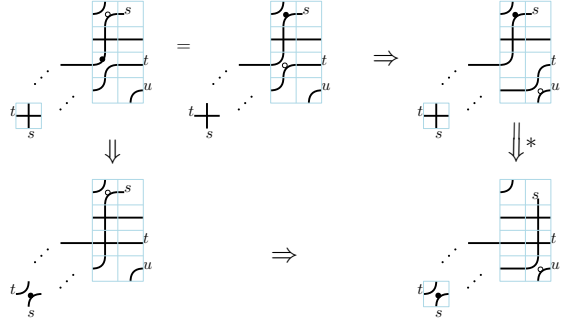}
    \captionof{figure}{$s$ and $t$ cross. Let $(d,e):=\droop_E(c,b)$. Notice that in column $b$ there must be a $\jt$ at or below row $d$, since $s$ and $t$ cross must cross below row $c$. If $E(d,e)=\rt$, let $u$ be the pipe that contains this $\rt$-turn. If $\pi^{-1}(u)>d_2$, then there must be no $\jt$ at or below row $\pi^{-1}(u)\le d$. This is a contradiction, so it must be the case that $\pi^{-1}(u)\le d_2\le k$. This means that the left move from $(E,\red,\yellow,u)$ must not be terminal.}
    \label{fig:Vstcross}
\end{minipage}

\vskip 2em
\noindent%
\begin{minipage}{\linewidth}
    \centering
    \includegraphics[scale=0.6]{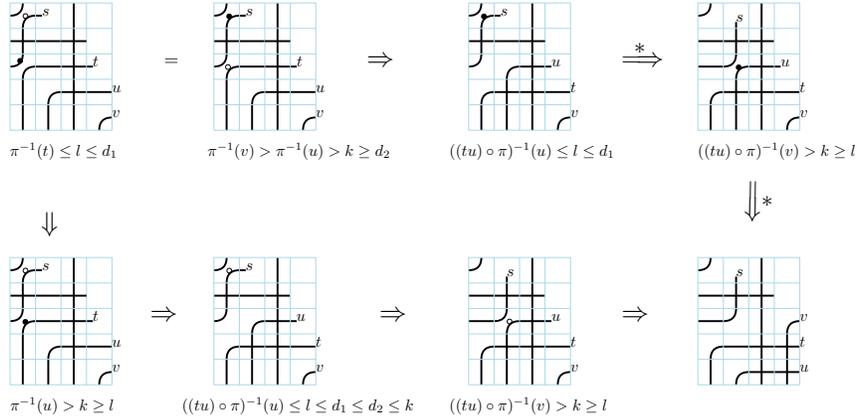}
    \captionof{figure}{$s$ and $t$ do not cross, left move is terminal, right move is not terminal. Let $(d,e):=\droop_E(c,b)$. Since the left move is terminal, it must be that $E(d,e)=\rt$ and  the pipe $u$ that contains this $\rt$-turn satisfies $\pi^{-1}(u)>k$. Furthermore, there are no $\jt$ or $\bl$ at or below row $d$. Since the right move is not terminal, it must be that $\pi^{-1}(t)\le l$. Let $v$ be the pipe that contains the $\rt$-turn at $E(\droop_E(d,e))$. The diagram shows that the states must converge in the terminal state.}
    \label{fig:Vst-nocross-lterm-rnoterm}
\end{minipage}

\vskip 2em
\noindent%
\begin{minipage}{\linewidth}
    \centering
    \includegraphics[scale=0.6]{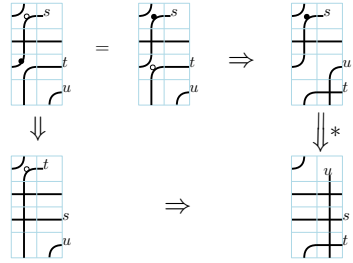}
    \captionof{figure}{$s$ and $t$ do not cross, both left and right moves are terminal.}
    \label{fig:Vst-nocross-lterm-rterm}
\end{minipage}

\vskip 1em
\noindent%
\begin{minipage}{\linewidth}
\centering
    \includegraphics[scale=0.6]{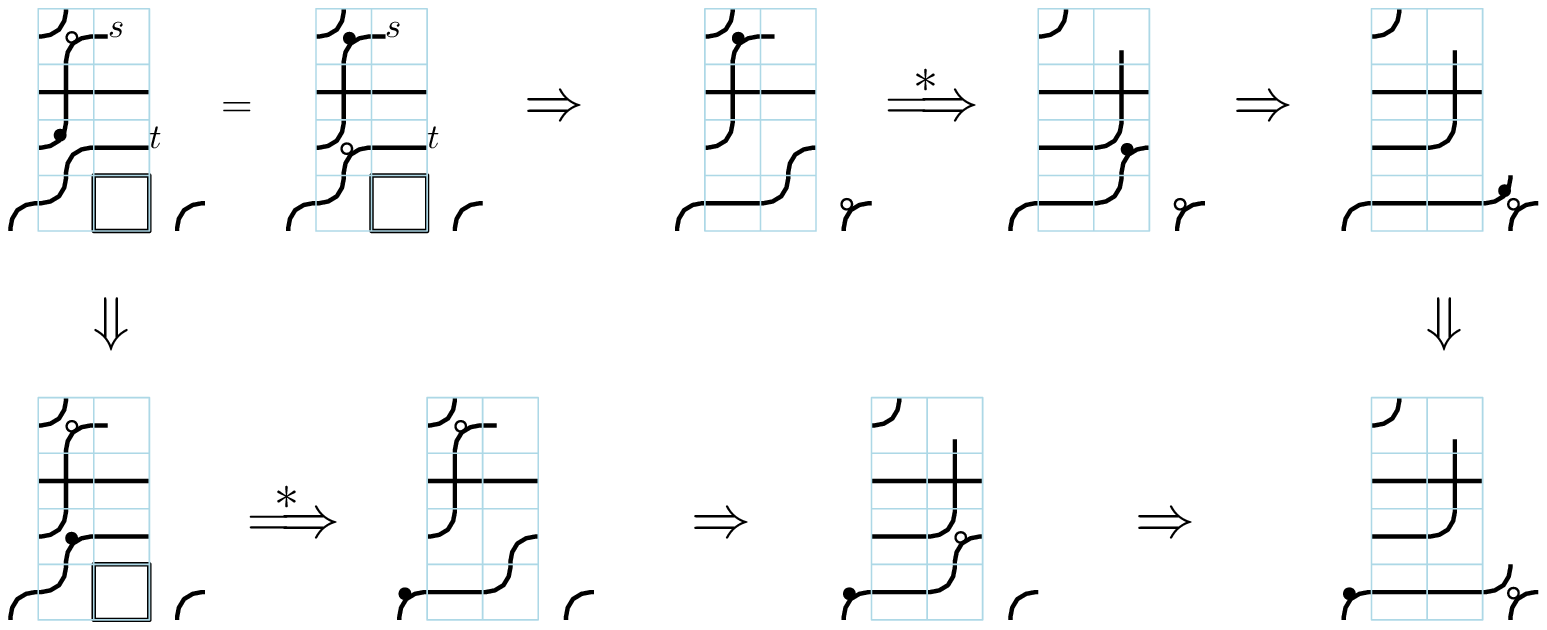}
    \captionof{figure}{$s$ and $t$ do not cross, neither left nor right move is terminal, and $E(\droop_E(c,b))=\bl$}
    \label{fig:Vst-nocross-noterm1}
\end{minipage}

\vskip 1em
\noindent%
\begin{minipage}{\linewidth}
\centering
    \includegraphics[scale=0.6]{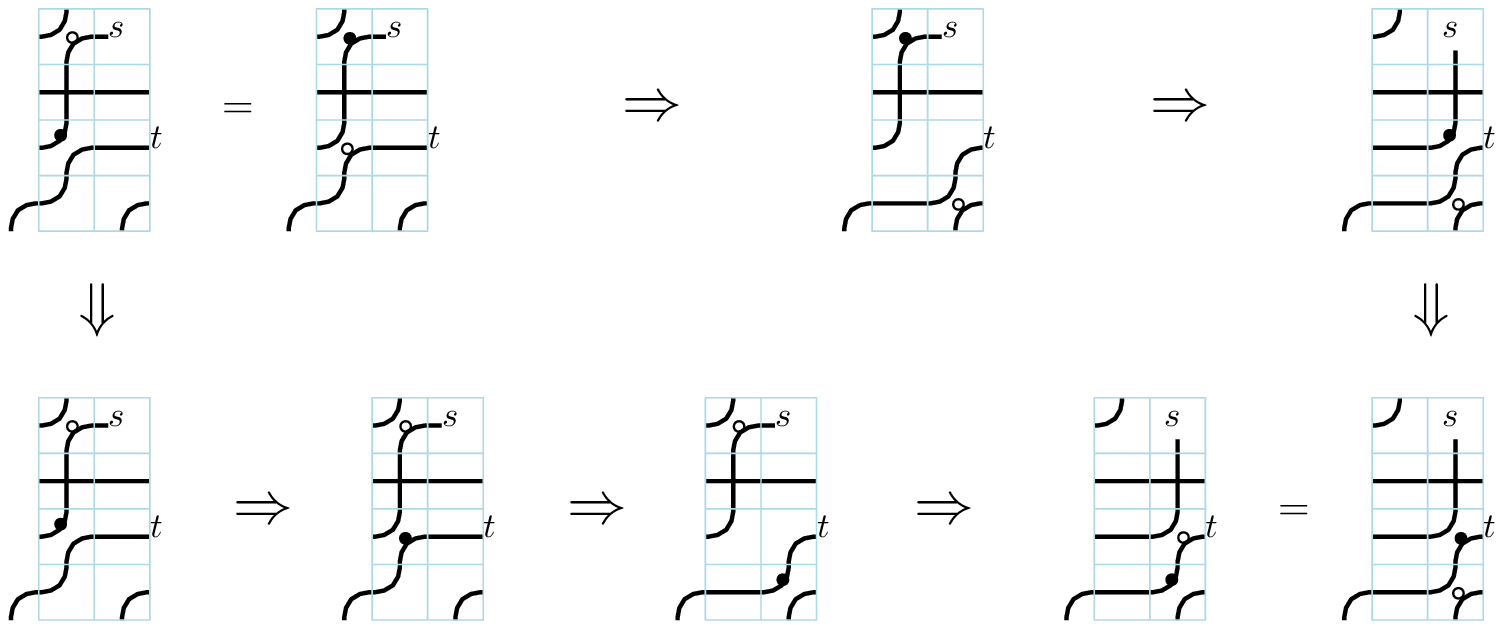}
    \captionof{figure}{$s$ and $t$ do not cross, neither left nor right move is terminal, and $E(\droop_E(c,b))=\rt$}
    \label{fig:Vst-nocross-noterm2}
\end{minipage}

\vskip 1em
\noindent%
\begin{minipage}{\linewidth}
\centering
    \includegraphics[scale=0.6]{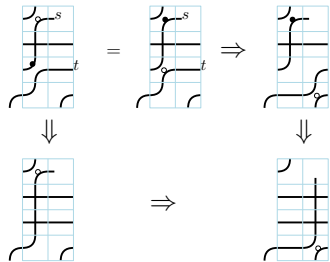}
    \captionof{figure}{$s$ and $t$ do not cross, left move is non-terminal, right move is terminal.}
    \label{fig:Vst-nocross-lnoterm-rterm}
\end{minipage}
\vskip 1em    
\textbf{Case ($\red=\yellow$).} Suppose $\red=\yellow=(a,b)$. Let $s$ be the pipe that contains the $\rt$-turn at $(a,b)$. Let $(c,d)=\droop_E(a,b)$. By forbidden configuration (3), $d=b+1$. We organize the cases as follows:
\begin{itemize}
    \item $E(c,d)=\bl$ (Figure~\ref{fig:C-bl})
    \item $E(c,d)=\rt$: let $t$ be the pipe that contains the $\rt$-turn at $(c,d)$
    \begin{itemize}
        \item $s$ and $t$ intersect
        \item $s$ and $t$ do not intersect
        \begin{itemize}
            \item Neither left nor right move is terminal
            \begin{itemize}
                \item $E(\droop_E(c,d))=\bl$ (Figure~\ref{fig:C-nocross-noterm})
                \item $E(\droop_E(c,d))=\rt$ (Figure~\ref{fig:C-nocross-noterm2})
            \end{itemize}
            \item Left move is terminal. In this case right move must also be terminal. (Figure~\ref{fig:C-nocross-lterm})
            \item Left move is not terminal, but right move is terminal.
        \end{itemize}
    \end{itemize}
\end{itemize}

\vskip 1em
\noindent%
\begin{minipage}{\linewidth}
\centering
    \includegraphics[scale=0.65]{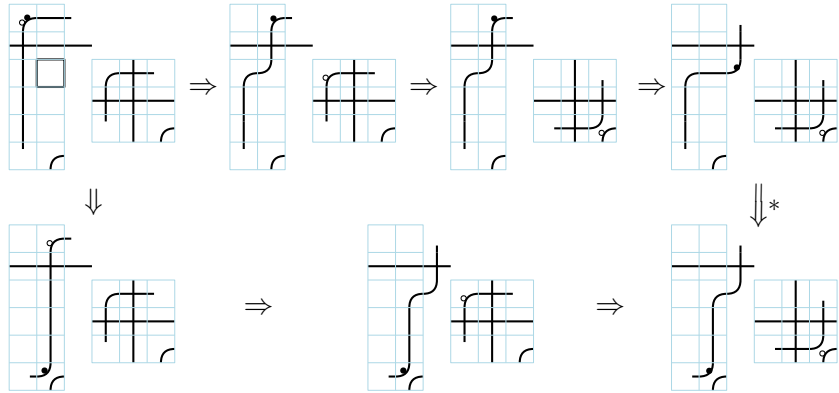}
    \captionof{figure}{$E(c,d)=\bl$. Notice that the state at which the moves converge could also be the result of a terminal left move, but this does not affect commutativity.}
    \label{fig:C-bl}
\end{minipage}

\vskip 1em
\noindent%
\begin{minipage}{\linewidth}
\centering
    \includegraphics[scale=0.6]{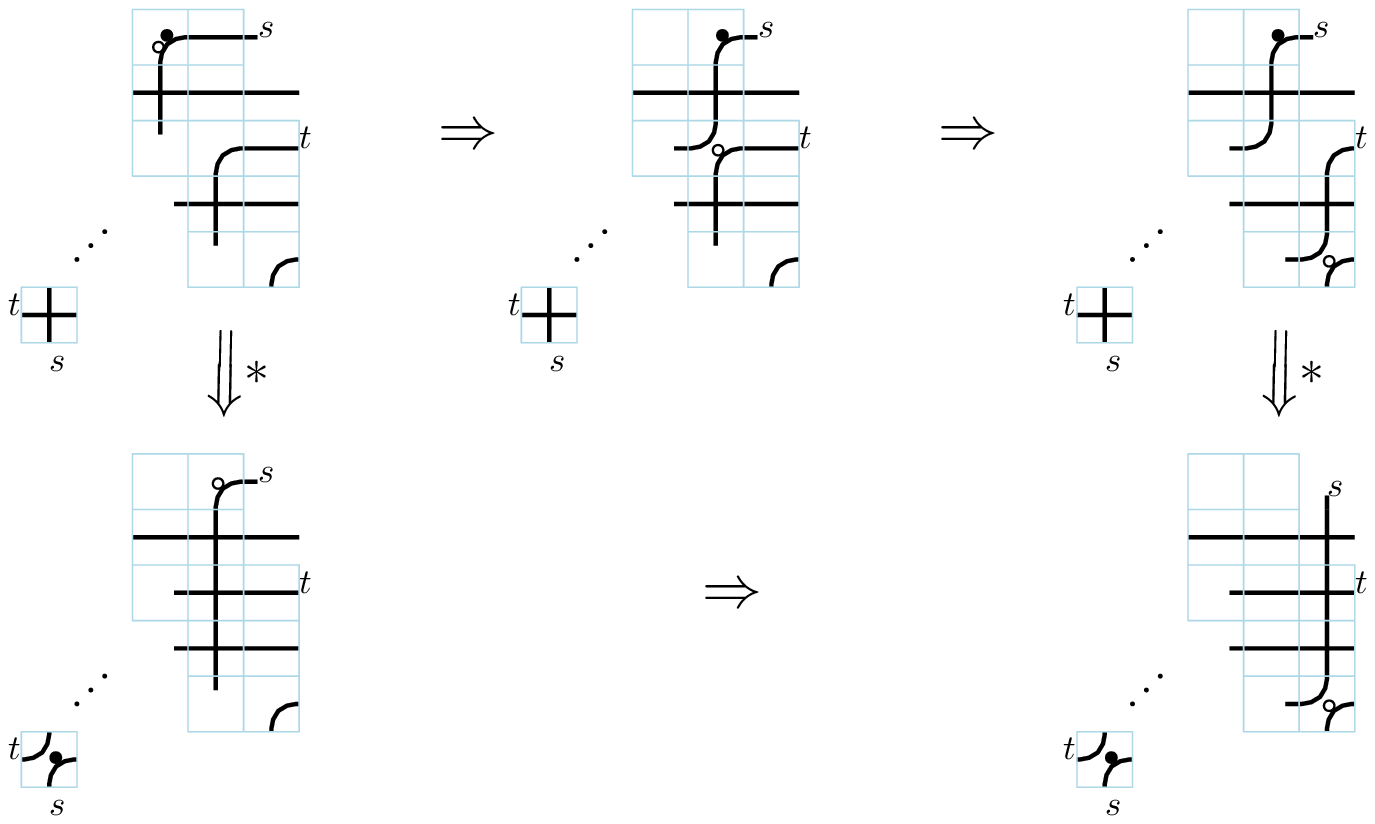}
    \captionof{figure}{$E(c,d)=\rt$, $s$ and $t$ cross}
    \label{fig:C-stcross}
\end{minipage}

\vskip 1em
\noindent%
\begin{minipage}{\linewidth}
\centering
    \includegraphics[scale=0.6]{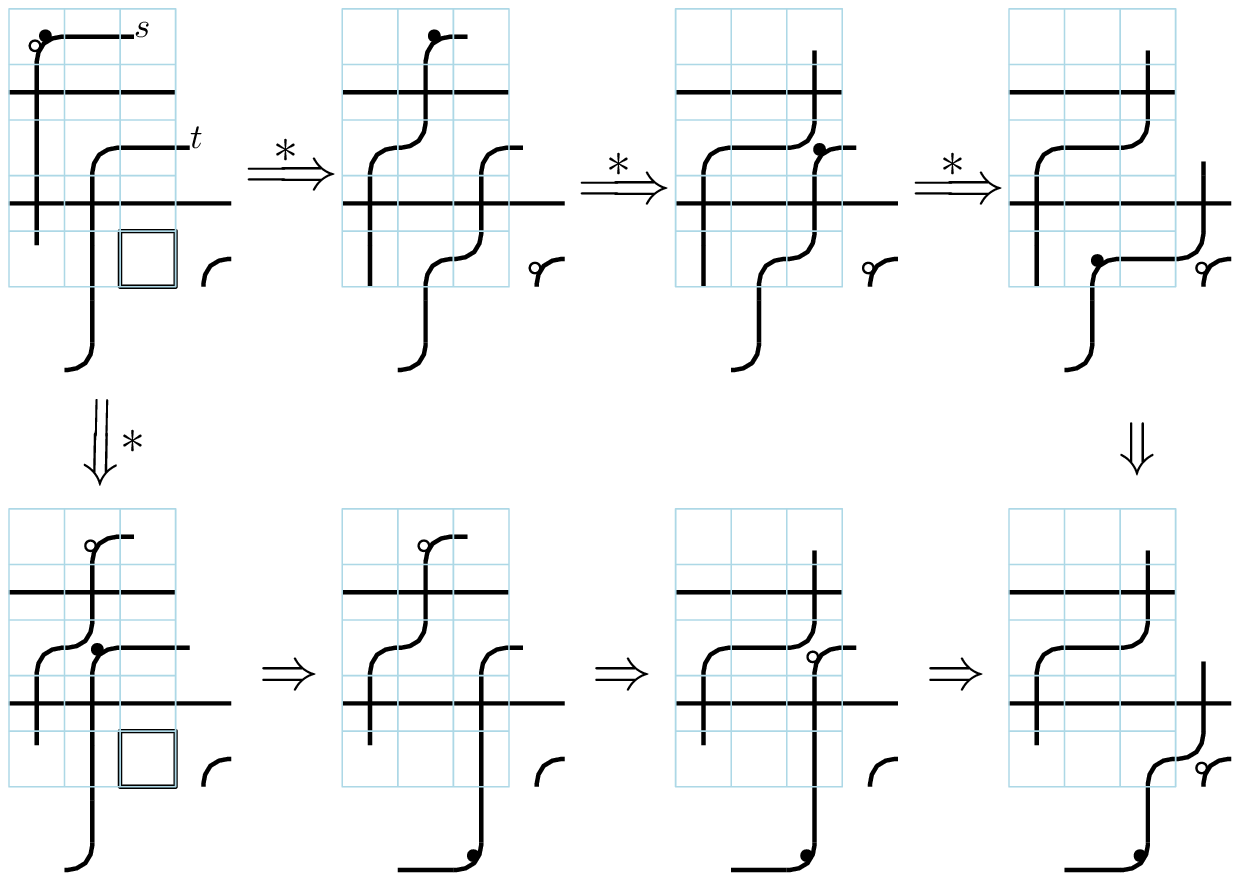}
    \captionof{figure}{$s$ and $t$ do not intersect, neither left nor right move is terminal, $E(\droop_E(c,d))=\bl$}
    \label{fig:C-nocross-noterm}
\end{minipage}

\vskip 1em
\noindent%
\begin{minipage}{\linewidth}    \centering
    \includegraphics[scale=0.6]{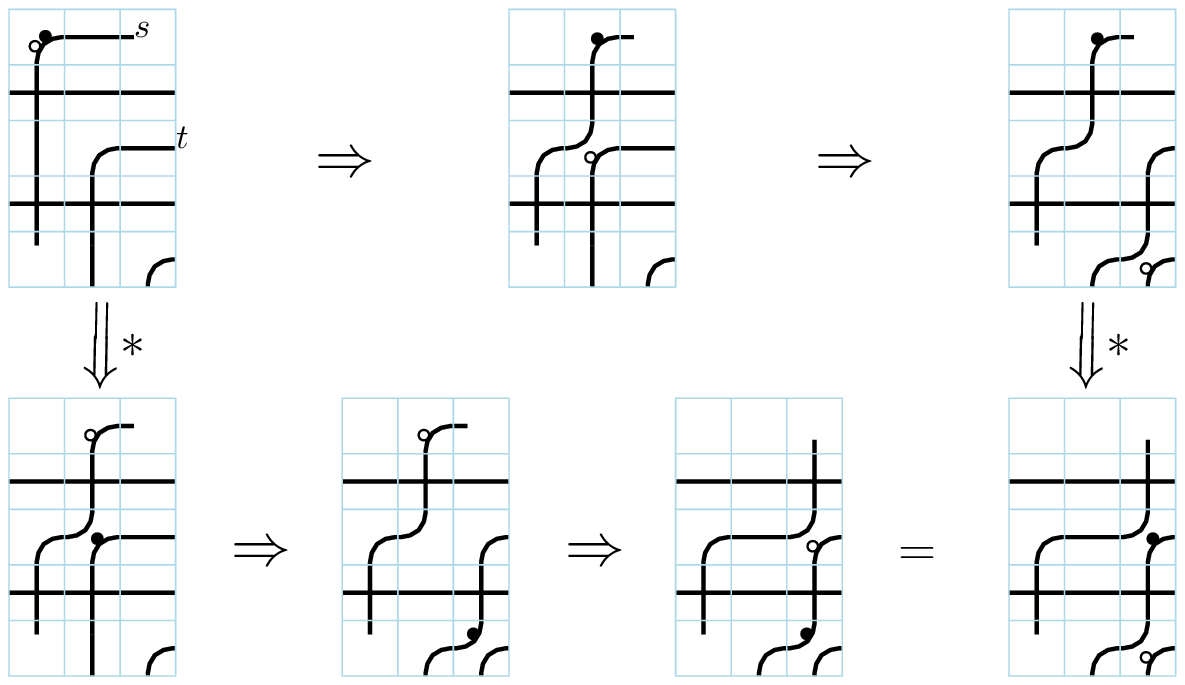}
    \captionof{figure}{$s$ and $t$ do not intersect, neither left nor right move is terminal, $E(\droop_E(c,d))=\rt$}
    \label{fig:C-nocross-noterm2}
\end{minipage}

\vskip 1em
\noindent%
\begin{minipage}{\linewidth}    \centering
    \includegraphics[scale=0.6]{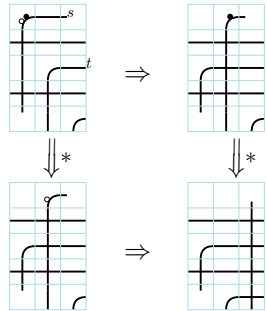}
    \captionof{figure}{$s$ and $t$ do not intersect, left move is terminal. This means $\pi^{-1}(t)>k\ge d_2$, which forces right move to be terminal.}
    \label{fig:C-nocross-lterm}
\end{minipage}

\vskip 1em
\noindent%
\begin{minipage}{\linewidth}    \centering
    \includegraphics[scale=0.6]{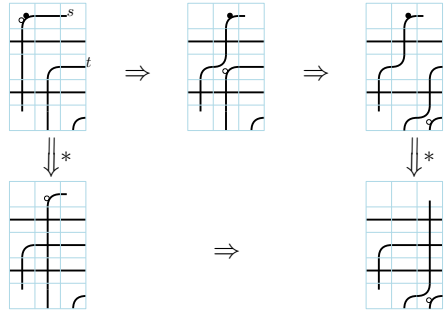}
    \captionof{figure}{$s$ and $t$ do not cross, left move is not terminal, right move is terminal. Notice that the state at which the moves converge could also be the result of a terminal left move, but this does not affect commutativity.}
\end{minipage}
\vskip 1em
\textbf{Case ($\red$ NW of $\yellow$).} Suppose $\red = (a,b)$ and $\maxd(a,b)=\yellow$. In all other situations when $\red$ is NW of $\yellow$, the move at $\red$ and the move at $\yellow$ are disjoint.
Let $(a_1,b_1)=\maxd(a,b)$, and $(a_2,b_2)=\droop(a_1,b_1)$. It must be the case that $b_1=b+1$. We organize the cases as follows:
\begin{itemize}
    \item $E(a_1,b)=\vtile$ 
    \begin{itemize}
        \item $E(a_2,b_1)=\vtile$ (Figure~\ref{fig:NW-1})
        \item $E(a_2,b_1)=\jt$ (Figure~\ref{fig:NW-2})
    \end{itemize}
    \item $E(a_1,b)=\jt$
\end{itemize}

\vskip 1em
\noindent%
\begin{minipage}{\linewidth}
\centering
    \includegraphics[scale=0.55]{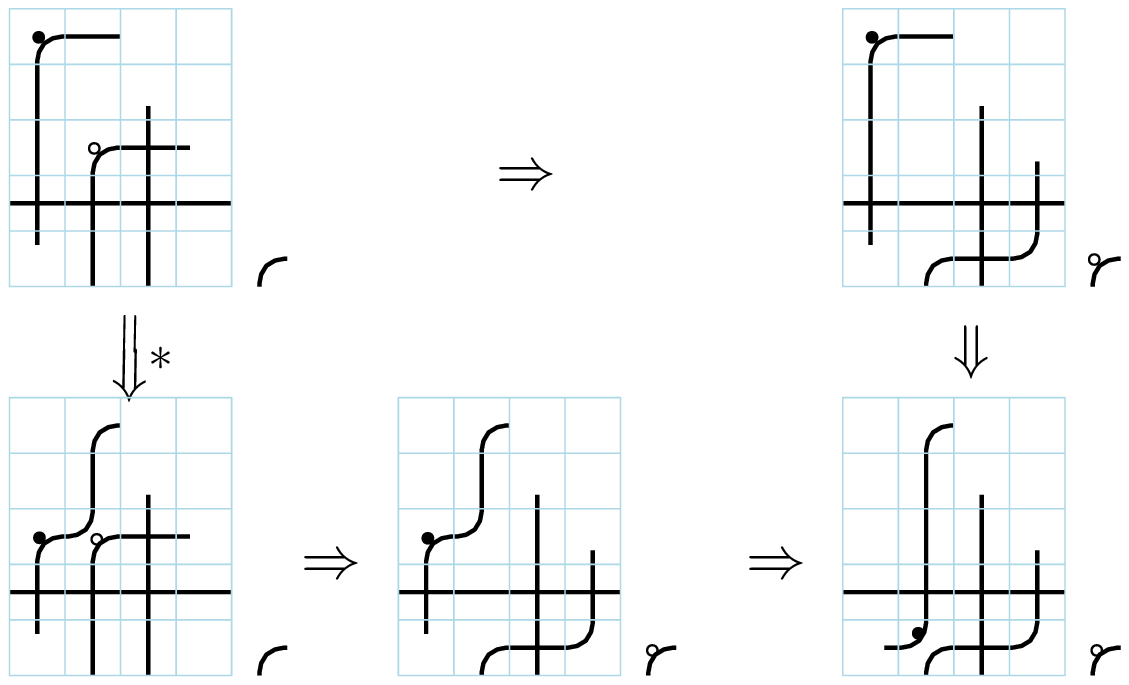}
    \captionof{figure}{$E(a_1,b)=\vtile$, $E(a_2,b_1)=\vtile$}
    \label{fig:NW-1}
\end{minipage}

\vskip 1em
\noindent%
\begin{minipage}{\linewidth}
\centering
    \includegraphics[scale=0.55]{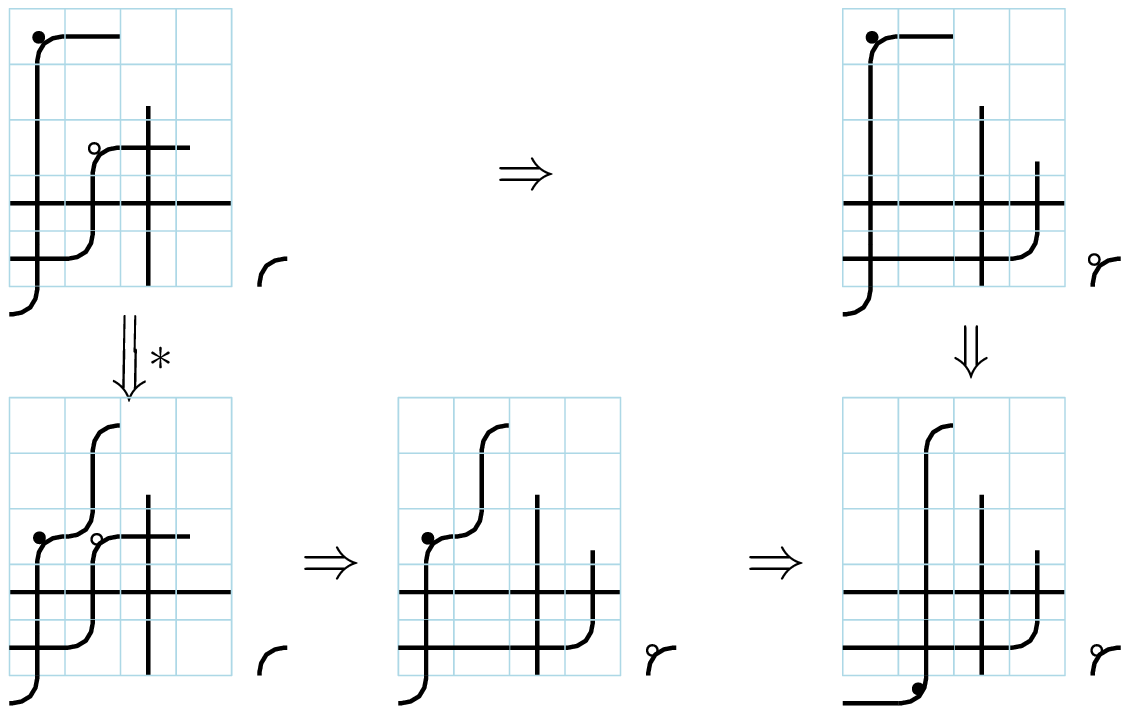}
    \captionof{figure}{$E(a_1,b)=\vtile$, $E(a_2,b_1)=\jt$}
    \label{fig:NW-2}
\end{minipage}

\vskip 1em
\noindent%
\begin{minipage}{\linewidth}
\centering
    \includegraphics[scale=0.55]{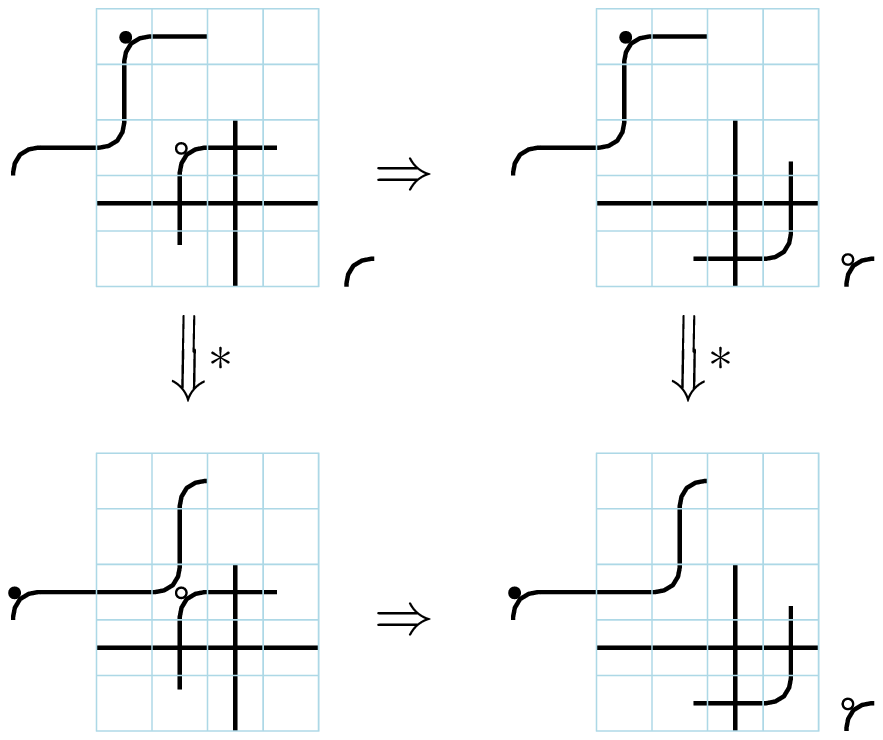}
    \captionof{figure}{$E(a_1,b)=\jt$}
    \label{fig:NW-3}
\end{minipage}
\vskip 1em
Case ($\red$ NE of $\yellow$). We suppose $\droop(\yellow)=\maxd(\red)=(a,b)$. In all other cases, the move at $\red$ and the move at $\yellow$ affect disjoint sets of tiles.  Supose $s$ is the pipe that passes through $\yellow$, $t$ is the pipe that passes through $\red$, and $u$ is the pipe that passes through $(a,b)$. We organize the cases as follows:
\begin{itemize}
    \item $E(a,b)=\bl$ (Figure~\ref{fig:NE-bl})
    \item $E(a,b)=\rt$
        \begin{itemize}
            \item Left move is not terminal (Figure~\ref{fig:NE-r-noterm})
            \item Left move is terminal (Figure~\ref{fig:NE-r-term})
        \end{itemize}
\end{itemize}

\vskip 1em
\noindent%
\begin{minipage}{\linewidth}    \centering
    \includegraphics[scale=0.52]{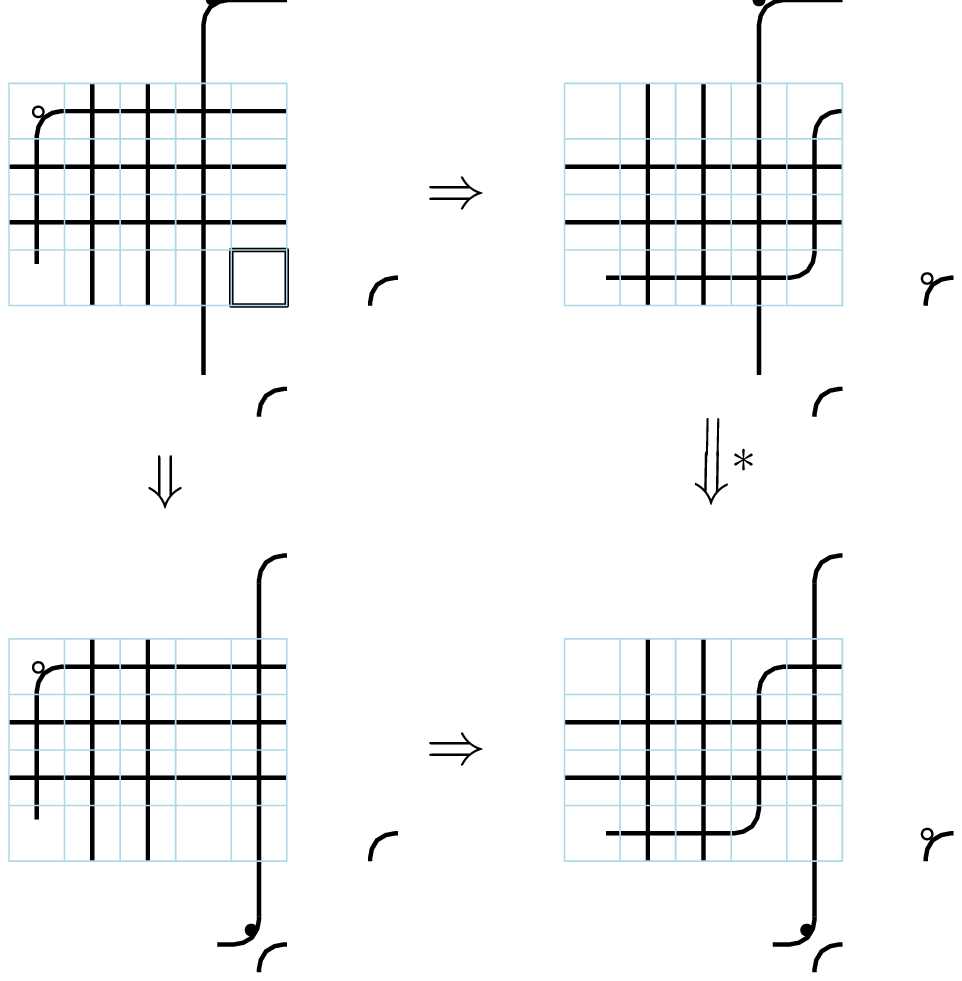}
    \captionof{figure}{$E(a,b)=\bl$}
    \label{fig:NE-bl}
\end{minipage}

\vskip 1em
\noindent%
\begin{minipage}{\linewidth}
\centering
    \includegraphics[scale=0.52]{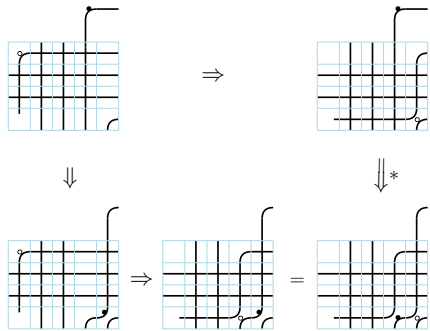}
    \captionof{figure}{$E(a,b)=\rt$. Left move is not terminal.}
    \label{fig:NE-r-noterm}
\end{minipage}

\vskip 1em
\noindent%
\begin{minipage}{\linewidth}
    \centering
    \includegraphics[scale=0.6]{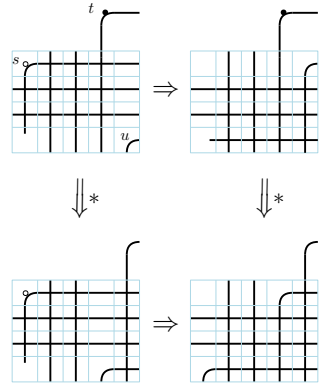}
    \captionof{figure}{$E(a,b)=\rt$. Left move is terminal. In this case, $\pi^{-1}(u)>k\ge d_2$. Then, $((su)\circ\pi)^{-1}(u)=\pi^{-1}(s)$. Since $s$ and $t$ cross in $E$, it must be the case  that $\pi^{-1}(s)>d_1\ge l$. Therefore, $((su)\circ\pi)^{-1}(u)>l$ and the downward arrow on the right in the diagram is correct.}
    \label{fig:NE-r-term}
\end{minipage}
\vskip 1em

\textbf{Case ($\red$ SE of $\yellow$)}. We can assume $\red=\droop(\yellow)$, otherwise the two moves affect disjoint sets of tiles (Figure~\ref{fig:SE}).
\vskip 1em
\noindent%
\begin{minipage}{\linewidth}
   \centering
    \includegraphics[scale=0.6]{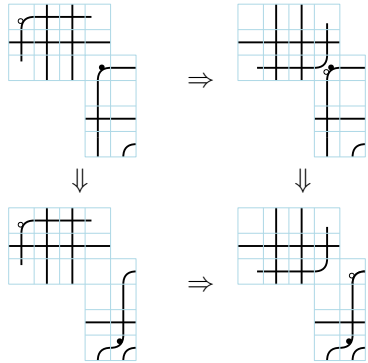}
    \captionof{figure}{$\red$ SE of $\yellow$. The left move (top arrow in the diagram) cannot be a terminating move, because the pipe that passes through $\red$ in $E$ must exit from a row $\le l \le k$.}
    \label{fig:SE}
\end{minipage}
\vskip 1em
\textbf{Case ($\red$ SW of $\yellow$)}. In this case it's impossible to have $\droop_E(\yellow)=\maxd_E{\red}$, because this implies the pipe $s$ that passes through $\red$ in $E$ and the pipe $t$ that passes through $\yellow$ in $E$ must cross at a tile where $s$ contains the horizontal segment, which is forbidden. Therefore, the moves at $\red$ and $\yellow$ involve disjoint tiles and must commute.
\end{proof}
\bibliographystyle{alpha}
\bibliography{ref}

\end{document}